\documentclass[a4paper,11pt]{article}
\usepackage{amsmath}
\usepackage{amssymb}
\usepackage{theorem}
\usepackage{euscript}
\usepackage{pstricks}
\usepackage{authblk}
\usepackage{relsize}

\usepackage{hyperref}
\hypersetup{
	colorlinks,
	citecolor=green,
	filecolor=black,
	linkcolor=magenta,
	urlcolor=magenta
}
\hypersetup{linktocpage}

\usepackage{color}

\def\II{{\mathbb I}}

\def\ZZ{{\mathbb Z}}
\def\NN{{\mathbb N}}
\def\RR{{\mathbb R}}

\def\Ff{{\mathcal F}}

\def\IId{{\mathbb I}^d}

\def\NNd{{\mathbb N}^d}
\def\RRd{{\mathbb R}^d}

\def\NNn{{\mathbb N}_{-1}}
\def\NNdn{{\mathbb N}^d_{-1}}

\def\Ff{{\mathcal F}}
\def\Ss{{\mathcal S}}

\def\Uas{\mathring{U}^{\alpha,d}_\infty}

\def\Uasone{\mathring{U}^{\alpha,1}_\infty}
\def\Uastwo{\mathring{U}^{\alpha,2}_\infty}

\def\Lad{H^{\alpha}_{\infty}(\II^d)}

\def\supp{\operatorname{supp}}

\newtheorem{theorem}{Theorem}[section]
\newtheorem{lemma}[theorem]{Lemma}
\newtheorem{corollary}[theorem]{Corollary}
\numberwithin{equation}{section}

\theoremstyle{definition}

\theoremstyle{remark}

\newcommand{\Chi}{\raise .3ex
	\hbox{\large $\chi$}}

\newcommand{\R}{\mathbb{R}}

\usepackage[section]{algorithm}
\usepackage{algorithmicx}
\usepackage{algpseudocode}
\algrenewcommand\algorithmicrequire{\makebox[46pt][l]{\textrm{required:}}}
\algrenewcommand\algorithmicensure{\makebox[46pt][l]{\textrm{output:}}}
\algrenewcommand\algorithmicfunction{\textrm{function}}
\algrenewcommand\algorithmicwhile{\textrm{while}}
\algrenewcommand\algorithmicdo{}
\algrenewcommand\algorithmicend{\textrm{end}}
\algrenewcommand\algorithmicforall{\textrm{for all}}
\algrenewcommand\algorithmicfor{\textrm{for}}
\algrenewcommand\algorithmicrepeat{\textrm{repeat}}
\algrenewcommand\algorithmicuntil{\textrm{until}}
\algrenewcommand\algorithmicif{\textrm{if}}
\algrenewcommand\algorithmicthen{\textrm{then}}
\algrenewcommand\algorithmicelse{\textrm{else}}

\newcommand{\ba}{{\boldsymbol{a}}}

\newcommand{\bee}{{\boldsymbol{e}}}

\newcommand{\bk}{{\boldsymbol{k}}}
\newcommand{\bh}{{\boldsymbol{h}}}

\newcommand{\bs}{{\boldsymbol{s}}}
\newcommand{\bt}{{\boldsymbol{t}}}
\newcommand{\bx}{{\boldsymbol{x}}}

\newcommand{\by}{{\boldsymbol{y}}}

\newcommand{\bxi}{{\boldsymbol{\xi}}}

\newcommand{\bell}{{\boldsymbol{\ell}}}

\newcommand{\btheta}{{\boldsymbol{\theta}}}

\newcommand{\blambda}{{\boldsymbol{\lambda}}}

\newcommand{\bone}{{\boldsymbol{1}}}

\usepackage{todonotes}

\newcommand{\be}{\begin{equation}}
\newcommand{\ee}{\end{equation}}
\newcommand{\beq}{\begin{eqnarray}}
\newcommand{\beqq}{\begin{eqnarray*}}
\newcommand{\eeq}{\end{eqnarray}}
\newcommand{\eeqq}{\end{eqnarray*}}

\title{{ 
			High-dimensional nonlinear approximation by parametric manifolds in 
		  H\"older-Nikol'skii  spaces of mixed smoothness
	  }}
\author[a]{Dinh D\~ung}
\affil[a]{Vietnam National University, Hanoi, Information Technology Institute
	\protect\\
	144 Xuan Thuy, Cau Giay, Hanoi, Vietnam
	\protect\\
	Email: dinhzung@gmail.com
}

\author[b]{Van Kien Nguyen}
\affil[b]{Faculty of Basic Sciences, University of Transport and Communications
	\protect\\	No.3 Cau Giay Street, Lang Thuong Ward, Dong Da District,
	Hanoi, Vietnam
	\protect\\
	Email: kiennv@utc.edu.vn}
\date{\today}
\begin{document}
	
\maketitle

\begin{abstract}
		 We study high-dimensional nonlinear approximation of functions in  H\"older-Nikol'skii spaces $\Lad$ on the unit cube $\IId:=[0,1]^d$ having  mixed smoothness, by parametric manifolds.  The approximation error is measured in the $L_\infty$-norm. In this context, we explicitly  constructed methods of nonlinear approximation, and give 
		dimension-dependent estimates of  the  approximation error explicitly in dimension $d$ and  number $N$ measuring computation complexity of the  parametric  manifold of approximants. For  $d=2$, we derived a novel right asymptotic order of noncontinuous manifold $N$-widths of the unit ball of $H^\alpha_\infty(\II^2)$ in the space $L_\infty(\II^2)$.
		In constructing approximation methods, the function decomposition by the tensor product Faber series and special representations of its truncations on sparse grids play a central role.



\end{abstract}

\section{Introduction}	

%
%

Some problems in approximation theory and numerical analysis driven by  a lot of applications in  Information Technology, Mathematical Finance, Chemistry, Quantum Mechanics, Meteorology, and, in particular, in Uncertainty Quantification and Deep Machine Learning, are formulated in  high dimensions when the number of involved variables are very large.   Numerical methods for such problems may require computational cost increasing exponentially in dimension  which makes the computation  intractable when the dimension of input data is very large. 
Hyperbolic crosses and sparse grids promise to rid this ``curse of dimensionality" in some  problems when  high-dimensional data belongs to certain classes of functions having  mixed smoothness. Function spaces having mixed smoothness appear naturally in many models of real world problem in mathematical physics,  finance and other fields, for instance,  the regularity properties  eigenfunctions of the electronic Schr\"odinger operator  \cite{Yser10} or the existence of solution of Navier-Stokes equations when initial data belonging to spaces with mixed smoothness \cite[Chapter 6]{Tri15B}. Approximation methods and sampling algorithms for functions having mixed smoothness constructed on  hyperbolic crosses and  sparse grids   give a surprising effect since hyperbolic crosses and sparse grids have the number of elements much less than those of standard domains and grids but give the same approximation error. This essentially reduces the computational cost, and therefore makes the problem tractable. Sparse grids for  approximate sampling recovery and integration were first considered by Smolyak	\cite{Smo63}.
In computational mathematics, the sparse grid approach was  initiated by Zenger  
\cite{Zen91}. There has been a very large number of papers on hyperbolic cross and sparse-gird approximation and numerical applications  to count all of them. We refer the reader to \cite{BuGr04,DTU18B} for surveys and for recent further developments and results. We also refer to the monographs \cite{NoWo08,NoWo10} for concepts and results on high dimensional problems and computation complexity.

 Let  us mention some    recent results on different aspects of the problem of dimension-dependent error estimation in high-dimensional approximation  which are directly related to our papers. 	
	The papers \cite{ChD16,CKS16,CKS19,DG16,DU13,DGHR18,KMU16,KSU15} are on this problem for hyperbolic cross approximation of functions with mixed smoothness in terms of various $n$-widths and $\varepsilon$-dimensions. The authors of \cite{DG16,DGHR18} in particular, extended these problems for infinite-dimensional approximation with applications to stochastic and parametric PDEs.  Preasymptotic estimation of  high-dimensional problems were also treated in \cite{KMU16,KSU14,KSU15}.  Related high-dimensional problems were studied in  \cite{PS19,TV19} based on ANOVA decomposition.  	
The paper \cite{DTU18B} has investigated dimension-dependent estimates of the approximation error for linear algorithms of
sampling recovery on Smolyak grids of functions from the space with H\"older–Zygmund mixed smoothness.  It proved some
upper bounds and lower bounds of the error of the optimal sampling recovery on Smolyak grids, explicit in dimension.	All of the above mentioned papers considered only linear problems of high-dimensional approximation.

While linear methods utilize approximation from finite-dimensional spaces,  nonlinear approximation means that the approximants do not come from linear spaces but rather from
sets of nonlinear structure such as nonlinear manifolds, set of finite cardinality,$\ldots$  It is well understood that nonlinear methods of approximation and numerical methods derived from them often produce superior performance when compared with linear methods. Several notions of linear and nonlinear widths have been introduced to quantify optimality of approximation methods. Let us recall some of them.

Let $X$ be a normed space, $F$ and $G$  subsets in $X$. We consider the problem of approximation of  $f \in F$ by elements $g \in G$. 
The approximation error is measured by $\|f - g\|_X$. 
The worst case  error of the approximation of elements $f \in F$ by elements $g \in G$ is defined as 
\begin{equation}\nonumber
E(F,G,X) := \ \sup_{f \in F} E(f,G,X)
:= \ \sup_{f \in F} \inf_{g \in G} \|f - g\|_X.
\end{equation}
In numerical applications, an (linear and nonlinear) approximation method is usually based on  a finite information in the form of the $N$ values  $b_1(f),\ldots, b_N(f)$ of functionals. Such an approximation method can be seen as
\begin{equation}\label{Q_N}
Q_N(f) = P_N(a_N(f)) \quad \text{for a pair of mappings} \ \ a_N:\, F \to \R^N \ \ \text{and} \  \ P_N: \, \R^N \to X.
\end{equation} 
The approximant set $G_N:= P_N(\RR^N)$ can be seen as a manifold in $X$ parameterized by $\RR^N$. The parameter $N$ characterizes computation complexity  of the approximation method. We specify  approximation methods having common properties by a certain set ${\mathcal Q}_N$ of pairs $(a_N,P_N)$, and look for an optimal method $Q_N \in {\mathcal Q}_N$ of approximation of $f \in F$ in terms of the quantity
\begin{equation}\label{E_N}
 d(F,{\mathcal Q }_N,X):= \inf_{(a_N,P_N) \in {\mathcal Q}_N} \ \sup_{f \in F}  \|f - P_N(a_N(f))\|_X.
\end{equation}
It is remarkable that the definition \eqref{E_N} is fit to notion of some important quantities of best linear and nonlinear approximation.
Thus, if linear approximation is understood as approximation  by elements from a finite-dimensional linear subspace, the well-known Kolmogorov widths $d_N(F,X)$ and linear $N$-widths   $\lambda_N(F,X)$  being  different quantities of best  linear approximation can be defined  as
\begin{equation}\label{d_N}
d_N(F,X)
:= \
d(F,{\mathcal Q }_N^d,X)
\ = \
\inf_{\text{linear subspaces }   X_N\atop \text{dim} X_N \le N}\ E(F,X_N),
\end{equation}
and 
$\lambda_N(F,X):= d(F,{\mathcal Q }_N^{\lambda},X)$,
where ${\mathcal Q }_N^{d}$ is the set of all pairs of  mappings $(a_N,P_N)$ such that $P_N$ maps $\RR^N$ to  some linear subspace $X_N \subset X$ of dimension at most $N$, and ${\mathcal Q }_N^{\lambda}$ is the set of all pairs of  linear mappings $(a_N,P_N)$. 
Here the right-hand side of \eqref{d_N} is the traditional definition of Kolmogorov $N$-widths (see, e.g., \cite{DTU18B} for a traditional definition of  linear $N$-widths).

We next discuss the definition \eqref{E_N} for  some quantities of best nonlinear approximation.  
 The first  notion given in \cite{DHC89}, is (continuous) manifold $N$-width and defined by requiring $(a_N,P_N)$ to be continuous:
\begin{equation*}\label{delta_N}
\delta_N(F,X):= \ d(F,{\mathcal Q }_N^{\delta},X)
:= \ 
\inf_{(a_N,P_N)  \in {\mathcal Q }_N^{\delta}} \ \sup_{f \in F}  \|f - P_N(a_N(f))\|_X,
\end{equation*}
where ${\mathcal Q }_N^{\delta}$ is the set of all pairs of continuous mappings $(a_N,P_N)$.  Here the approximant set $G_N:= P_N(\RR^N)$ is a continuous manifold in $X$.


The requirement of continuity on $a_N,P_N$ is too minimal and does not give stability used in practice.  To have stability in the numerical implementation  one  can restrict  mappings $a_N$ and $P_N$ to be  Lipschitz continuous. Based on this idea, in \cite{CDPW20} the authors have introduced a notion of  stable manifold $N$-widths by the formula $\delta^*_N(F,X):= d(F,{\mathcal Q }_N^{\delta^*},X)$, where  ${\mathcal Q }_N^{\delta^*}$ is the subset in 
${\mathcal Q }_N^{\delta}$ of  all  Lipschitz mappings $a_N,P_N$ with some fixed constant $\gamma \ge 1$, that is 
$|a_N(f) - a_N(g)| \le \gamma  \|f - g\|_X$, and   $\|P_N(\bx) - P_N(\by)\|_X \le \gamma  |\bx - \by|$, $\bx, \by \in \RR^N $ with the Euclid norm $|\cdot|$.

However, in many numerical applications 
approximation methods do not have continuous properties.
The  nonlinear $N$-width which is not based on continuity condition was suggested by Kolmogorov (1955)  in the form of inverse quantity,  $\varepsilon$-entropy. This is entropy $N$-width 
\begin{equation*}\label{e_N}
\varepsilon_N(F,X)
:= \
d(F,{\mathcal Q }_N^\varepsilon,X)
\ = \
\inf_{ X_N \subset X:\ |X_N| \le 2^N}\ E(F,X_N),
\end{equation*}
where  ${\mathcal Q }_{N}^{\varepsilon}$ is the set of all pairs of  mappings $(a_N,P_N)$ such that $a_N$ maps $F$ into $\{0,1\}^N \subset \RR^N$  and $|X_N|$ denotes the cardinality of  $|X_N|$.

Another way to avoid the continuous restriction to require the approximant set $G_N:= P_N(\RR^N)$ which is in general, noncontinuous  manifold parameterized by $\RR^N$,  to be contained in a finite-dimensional linear subspace.
This leads to a notion of (noncontinuous) nonlinear manifold $N,M$-width $d_ {N,M}(F,X)$ of a  subset $F$ in $X$ as
\begin{equation*}\label{gamma_N}
d_ {N,M}(F,X)
:= \
d(F,{\mathcal Q }_{N,M}^{d},X)
\ = \
\inf_{(a_N,P_N) \in {\mathcal Q }_{N,M}^{d} } \ \sup_{f \in F}  \|f - P_N(a_N(f))\|_X,
\end{equation*}
where  ${\mathcal Q }_{N,M}^{d}$ is the set of all pairs of  mappings $(a_N,P_N)$ such that $P_N$ maps $\RR^N$ to  some linear subspace $X_M \subset X$ of dimension at most $M$. The parameter  $M$ in some sense only controls  the linear dimension of the parametric manifold $P_N(\RR^N)$, but is not related to computation complexity  of the approximation method which is as above mentioned, characterized by the parameter $N$. Notice that with $N \le M$ 
\begin{equation}\label{d_M<d_N,M}
d_M(F,X) \ \le \ d_ {N,M}(F,X) \ \le \ d_N(F,X) \quad \text{and} \quad d_ {N,N}(F,X)= d_N(F,X).
\end{equation}

 One may assume that $N$ and $M$ are comparable, in particular, take  $M = M(N)$  with the restriction  $N \le M(N) \le   C N (\log N)^\kappa$ for some $\kappa \ge 0$ and $C \ge 1$. With this assumption  $d_ {N,M(N)}(F,X)$ now only depends on $N$. It is surprising that 
for some cases,  $d_ {N,M(N)}(F,X)$ may have asymptotic order less than asymptotic order of any known nonlinear $N$-widths. This is confirmed at least  by the following example.

Let $\Uas$ be the unit ball of the H\"older-Nikol'skii space of functions on the unit cube $\II^d$ of mixed smoothness $0< \alpha \le 1$ with zero on the boundary of $\IId$ (see Section  \ref{sec:fabersystem} for a definition). Then when $d=1$ we have that 
\begin{equation}\label{delta_Nasymp}
d_N(\Uasone, L_\infty(\II))
\ \asymp \
\delta_N(\Uasone, L_\infty(\II))
\ \asymp \
N^{-\alpha},
\end{equation}
and 
\begin{equation}\label{gamma_Nasymp}
d_ {N,\lfloor N \log N\rfloor}(\Uasone, L_\infty(\II))
\ \asymp \
(N \log N)^{-\alpha}.
\end{equation}
The asymptotic order  of $d_N(\Uasone, L_\infty(\II))$ in \eqref{delta_Nasymp} is well-known, see, e.g. , \cite{Tem18B} for references. The asymptotic order  of $\delta_N(\Uasone, L_\infty(\II))$ in \eqref{delta_Nasymp} was proven in \cite{DHC89}. The upper bound of \eqref{gamma_Nasymp} follows from recent results obtained in \cite{Ya17b} for the case $\alpha = 1$ and in \cite{DDF.19} for the case $\alpha\in (0,1)$.  The lower bound of \eqref{gamma_Nasymp}  follows from the  inequalities
$$d_ {N,\lfloor N \log N\rfloor}(\Uasone, L_\infty(\II)) \ge d_{\lfloor N \log N\rfloor}(\Uasone, L_\infty(\II)) \asymp (N \log N)^{-\alpha}.$$ 
Here, we want to emphasize  that the right asymptotic order of Kolmogorov width $d_N(\Uas, L_\infty(\IId))$ and  linear $N$-widths $\lambda_N(\Uas, L_\infty(\IId))$ as well as of nonlinear $N$-widths for $d\ge 2$ is open problems except the case $d=2$ (see \cite[Chapter 4]{DTU18B} for detailed comments).

 All the above remarks and comments motivate us to consider    high-dimensional nonlinear approximation by parametric manifolds  for functions from the unit ball $\Uas$ of H\"older-Nikol'skii spaces having  mixed smoothness $\alpha$. The approximation error is measured in the $L_\infty$-norm.  In this context, we investigate the explicit construction of approximation methods of the form \eqref{Q_N}  with $(a_N,P_N) \in {\mathcal Q }_{N,M(N)}^{d}$ for approximation of $f \in \Uas$, and  explicit estimates in dimension $d$ and $N$ of the approximation error. We also treat the problem of  right asymptotic order of $d_ {N,M(N)}(\Uas,L_\infty)$. (Here and in what follows, we use the abbreviation: $L_\infty:=L_\infty(\IId)$ and $\|\cdot\|_\infty:= \|\cdot\|_{L_\infty}$.)

 Let us briefly describe our main contribution. 
 Let    $N\in \NN$ with 
$N\geq N(d)$ be given and $M(N):=\big\lfloor \frac{(12d^3)^{d-1}}{(d-1)!}N(\log N)(\log \log N)^{-(d-1)} \big\rfloor$ where $N(d)$ is a certain number (see \eqref{N>}). Then  we  can  explicitly construct 
a $M(N)$-dimensional subspace $F_{M(N)}$ of continuous   functions on $\IId$ spanned by tensor product Faber basis functions,  and maps
\begin{equation} \label{lambda_N,G_N}
\blambda^*_N: \Uas \to    \RR^N \qquad {\rm and}  \qquad G^*_N: \RR^N \to F_{M(N)}\subset C(\IId),
\end{equation}
so that 
\begin{align} \label{upperbnd}
	\sup_{f\in \Uas}\|f-G_N^*(\blambda^*_N(f))\|_\infty
\	\leq  \
	 C_\alpha \left(\frac{K^{d-1}}{(d-1)!}\right)^{2\alpha +1} \frac{(\log N)^{(d-1)(\alpha+1)}}{(N\log N)^{\alpha} }  (\log\log N)^{(d-1)\alpha},
\end{align}
and 
\begin{align} \label{lowerbnd}
\sup_{f\in \Uas}\|f-G_N^*(\blambda^*_N(f))\|_\infty
\ \geq \
C_{\alpha,d}	\frac{(\log N)^{(d-1)(\alpha +\frac{1}{2})}}{(N\log N)^{\alpha} }  (\log\log N)^{(d-1)\alpha},
\end{align}
where $K:= \left(4^\alpha 6/ (2^{\alpha}-1)\right)^{1/(2\alpha +1)}$, $C_\alpha:=   2^{7\alpha+2}/(2^{\alpha}-1)$ and  the constant $C_{\alpha,d}$ depends on $\alpha,d$ only.
 In the case $d=1$, \eqref{upperbnd} follows from results on approximation by deep ReLU networks which  have been proven in \cite{Ya17b} ($\alpha=1$) and  \cite{DDF.19} ($\alpha\in (0,1)$).  Notice the term $\big(\frac{K^{d-1}}{(d-1)!}\big)^{2\alpha +1}$  in the right-hand of \eqref{upperbnd} decays super-exponentially  when $d \to \infty$.  To our knowledge  \eqref{upperbnd} is the first result on  dimension-dependent error estimation of  nolinear approximation of functions having mixed smoothness.

From  \eqref{upperbnd} and \eqref{lowerbnd} we also derived some upper and lower  estimates for the noncontinuous manifold widths $d_ {N,M(N)} (\Uas,L_{\infty})$ (see Corollary \ref{cor:d_NM}).
Especially,  when  $d=2$ we obtain the novel right asymptotic order  
\begin{equation*}  
	\begin{aligned}
		d_ {N,\lfloor N(\log N)(\log \log N)^{-1}\rfloor} (\Uastwo,L_\infty(\II^2))
\asymp
	N^{-\alpha}	\log N  (\log\log N)^{\alpha}.
	\end{aligned}
\end{equation*}
(The case $d=1$ already is given in \eqref{gamma_Nasymp}.)
Let us compare this asymptotic order with the asymptotic order of other well-known $N$-widths.  We have  for $\alpha\in (0,1)$ that
\begin{equation}\label{eq-manifold}
	\begin{aligned}
&\varepsilon_N(\Uastwo,L_\infty(\II^2))
\asymp 
\delta^*_N(\Uastwo,L_\infty(\II^2))
\\
&
\asymp 
d_ N (\Uastwo,L_\infty(\II^2))
\asymp
\lambda_ N (\Uastwo,L_\infty(\II^2))
\asymp
N^{-\alpha} (\log N)^{\alpha+1}.
	\end{aligned}
\end{equation}
The results in \eqref{eq-manifold} were proven in \cite{Tem95} for entropy $N$-widths, and in \cite{Tem96} Kolmogorov $N$-widths.  For the asymptotic order of $\lambda_N(\Uastwo,L_\infty(\II^2))$ in \eqref{eq-manifold} see \cite[p. 67]{DTU18B}. The asymptotic order of $\delta^*_N(\Uastwo,L_\infty(\II^2))$ in \eqref{eq-manifold} follows from a Carl's type  inequality between  $\varepsilon_N$ and $\delta^*_N$ \cite{CDPW20} and  the  inequality  $\delta^*_N \le \lambda_N$. 
To  knowledge of the authors, the asymptotic order of  $\delta_ N(\Uastwo,L_\infty(\II^2))$ is not  known, except the upper bound via the inequality   $\delta_N \le d_N$ and \eqref{eq-manifold}.
Comparing the asymptotic order of  $d_ {N,\lfloor N(\log N)(\log \log N)^{-1}\rfloor} (\Uastwo,L_\infty(\II^2))$ with the asymptotic order of the ``smallest" entropy $N$-widths  $\varepsilon_N(\Uastwo,L_\infty(\II^2))$ and the other $N$-widths in \eqref{eq-manifold}, we find that the first one is smallest in logarithm scale.

 In construction of approximation and estimation of the approximation error, a representation of functions  in $\Uas$ by tensorized Faber series plays a central role. We primarily approximate $f\in \Uas$ by the truncations of tensorized Faber series $R_m(f)$ on sparse grids, and then approximate the function $f-R_m(f)$  by combining a sparse-grid interpolation approximation and an approximation by sets of finite cardinality.

The outline of this  paper is as follows. In Section  \ref{sec:fabersystem}, we present some auxiliary knowledge: a definition of H\"older-Nikol'skii  spaces of mixed smoothness $\Lad$ and a representation of functions in $\Lad$ based on tensorized Faber basis. In this section, we also study 
auxiliary approximation of  functions $f\in \Uas$ by  truncations of tensorized Faber series $R_m(f)$ on sparse grids,   and  approximation of $f$ by sets of finite cardinality. Section \ref{sec-nonlinear-manifold} is devoted to construction of  manifold approximation  for functions in H\"older-Nikol'skii spaces and estimation of the approximation error. 

\noindent
{\bf Notation.} \ As usual, $\NN$ is the natural numbers, $\ZZ$ is the integers, $\RR$ is  the real numbers and $ \NN_0:= \{s \in \ZZ: s \ge 0 \}$; $\NNn= \NN_0\cup \{ -1\} $. 
The letter $d$ is  reserved for
the underlying dimension of $\RR^d$, $\NN^d$, etc. We use  $x_i$ to denote the $i$th coordinate 
of $\bx \in \RR^d$, i.e., $\bx := (x_1,\ldots, x_d)$. For $\bx, \by \in \RR^d$,  $\bx \by $ denotes
the  Euclidean inner product of $\bx, \by$, and
$2^\bx := (2^{x_1},\ldots,2^{x_d})$. For $\bk, \bs \in \NNd_0$,  we denote $2^{-\bk}\bs := (2^{-k_1}s_1,\ldots,2^{-k_d}s_d)$.  For $\bx\in \RR^d$, we denote 
$|\bx|_1 := |x_1|+\ldots+|x_d|$. We use the abbreviation: $L_\infty:= L_\infty(\IId)$ and $\|\cdot\|_\infty:= \|\cdot\|_{L_\infty}$.
Universal constants or constants depending on parameter $\alpha, d$ are denoted by $C$ or $C_{\alpha,d}$, respectively.  Values of constants $C$ and  $C_{\alpha,d}$ in general, are not specified except the case when they are precisely given, and may be different in various places. For two sequences $a_n$ and $b_n$ we will write $a_n \lesssim b_n$ if there exists a
constant $C>0$ such that $a_n \leq C\,b_n$ for all $n$, and  $a_n \asymp b_n$ if $a_n \lesssim b_n$ and $b_n
\lesssim a_n$.  $|A|$ denotes the cardinality of the finite set $|A|$.


\section{ Approximation by truncated Faber series}\label{sec:fabersystem}

This section presents some preliminaries. We first
provide a definition of H\"older-Nikol'skii  spaces of mixed smoothness $\Lad$ and certain properties of these spaces.  As a preparation for the manifold approximation in the next section we recall a representation of continuous functions on the unit cube  by the tensorized Faber series.  We then give an estimate for the representation coefficients of functions   from  H\"older-Nikol'skii spaces and the error of the approximation of $f\in \Lad$ by truncations of the tensorized Faber series $R_m(f)$. In the last part of this section, a set of finite cardinality is  explicitly constructed to approximate functions in  $\Lad$ and approximation error is given explicitly in $d$.

\subsection{H\"older-Nikol'skii  spaces of mixed smoothness}
 This subsection is devoted to introducing the H\"older-Nikol'skii  spaces of mixed smoothness  under consideration. For univariate functions $f$ on $\II$, the difference operator $\Delta_h$ is defined by 
\begin{equation*}
\Delta_h f(x) := \
f(x + h) -  f(x),
\end{equation*}
for all $x$ and $h \ge 0$ such that $x, x + h \in \II$. 
If $u$ is a subset of $\{1,\ldots,d\}$, for multivariate functions $f$ on $\IId$
the  mixed  difference operator $\Delta_{\bh,u}$ is defined by 
\begin{equation*}
\Delta_{\bh,u} := \
\prod_{i \in u} \Delta_{h_i}, \quad \Delta_{\bh,\varnothing} = {\rm Id},
\end{equation*}
for all $\bx$ and $\bh$ such that $\bx, \bx +\bh \in \IId$. Here the univariate operator
$\Delta_{h_i}$ is applied to the univariate function $f$ by considering $f$ as a 
function of  variable $x_i$ with the other variables held fixed. 
If $0 < \alpha \le 1$, 
we introduce the semi-norm 
$|f|_{H^{\alpha}_{\infty}(u)}$ for functions $f \in C(\IId)$ by
\begin{equation} \nonumber
|f|_{H^{\alpha}_{\infty}(u)}:= \
\sup_{\bh > 0} \ \prod_{i \in u} h_i^{-\alpha}\|\Delta_{\bh,u}(f)\|_{C(\IId(\bh,u))}
\end{equation}
(in particular, $|f|_{H^{\alpha}_{\infty}(\varnothing)}= \|f\|_{C(\IId)}$), where 
$\IId(\bh,u):= \{\bx \in \IId: \, x_i + h_i \in \II, \, i \in u\}$.
The  H\"older-Nikol'skii space 
$\Lad$ of mixed smoothness $\alpha$ then is defined as the set of  functions $f \in C(\IId)$ 
for which the  norm 
\begin{equation*} 
\|f\|_{\Lad}
:= \ 
\max_{u \subset \{1,\ldots,d\}} |f|_{H^{\alpha}_{\infty}(u)}
\end{equation*}
is finite. 
From the definition we have that $\Lad \subset C(\IId)$. The space $\Lad$ is a $d$-time tensor product of the space $H_\infty^\alpha(\II)$ in the sense of equivalent norms. For further properties of this space such as embeddings, characterization by wavelets and atoms, we refer the reader to \cite{Ni75B,ST87B,Vy06} and references there. 

Denote by  $\mathring{C}(\IId)$ the set of all functions $f \in C(\IId)$ vanishing on the boundary $\partial  \IId$ of $\IId$, i.e., the set of all functions $f \in C(\IId)$ such that  $f(x) =0$ if  $x_j=0$ or $x_j=1$  for some index $j \in \{1,\ldots,d\}$.
Let $\Uas$ be the set of all functions $f$ in the intersection 
$\Lad \cap \mathring{C}(\IId)$  such that $\|f\|_{\Lad} \le 1$. 

\subsection{Tensorized Faber series and sparse-grid interpolation sampling recovery}
 In this subsection we describe  a representation
of functions in $\Lad$ by tensorized Faber series which plays a central role in the construction of nonlinear methods of noncontinuous manifold approximation of functions from the unit ball $\Uas$.  We give a dimension-dependent estimate of the approximation error by  truncation ${R}_m(f)$ of the tensorized Faber series for functions $f \in \Uas$. The approximant ${R}_m(f)$ represents an interpolation sampling recovery  on  sparse Smolyak grids.

We start with the univariate case. Let  $\varphi(x)\ = \ (1 - |x-1|)_+$, $x \in \II$, be  the hat function (the piece-wise linear B-spline with knots at $0,1,2$), where 
$x_+:= \max(x,0)$ for $x \in \RR$. 
For $k\in \NNn$ we define the Faber functions $\varphi_{k,s}$ by
\begin{equation*}\label{eq:faber1}
\varphi_{k,s}(x):=
\varphi(2^{k+1}x - 2s), \quad k \geq 0,  \ s \in Z(k):=\{0,1,\ldots, 2^{k} - 1\},
\end{equation*}
and
\begin{equation*}\label{eq:faber2}
\varphi_{-1,s}(x) := \varphi (x - s + 1),\ \ s\in Z(-1):=\{0,1\}.
\end{equation*}
For a univariate function $f$ on $\II$, 
$k \in \NNn$, and $s\in Z(k)$ we define
\begin{equation*} 
\lambda_{k,s}(f) \ := 
- \frac {1}{2} \Delta_{2^{-k-1}}^2 f\big(2^{-k}s\big),\ \ k \ge 0, \quad 
\lambda_{-1,s}(f) \ := f(s).
\end{equation*}
Here 
$$
\Delta_h^2 f(x) := \
f(x + 2h) -  2f(x+h)+f(x),
$$
for all $x$ and $h \ge 0$ such that $x, x + h \in \II$. 
The functions $\varphi_{k,s}$, $k \in \NN_{-1}$, $s \in Z(k)$, constitute a basis for $C(\II)$ and  every function 
$f \in C(\II)$ can be represented by the Faber series \cite{Fab09}
\begin{equation} \label{eq:FaberRepresentationd=1}
f
\ = \
\sum_{k \in \NN_{-1}} q_{k}(f), \ \qquad q_{k}(f) := \sum_{s\in Z(k)} \lambda_{k,s}(f)\varphi_{k,s}
\end{equation}
converging in the norm of $C(\II)$.

For $m\in \NN_0$, we define the truncation of the Faber series ${R}_m(f)$ by
\[{R}_m(f) := \sum_{k=0}^m q_k(f).
\]
The continuous piece-wise linear function ${R}_{m}(f) \in \mathring{C}(\II)$ possesses a certain interpolatory property. 
Indeed, one can check  that for $f \in \mathring{C}(\II)$
\begin{equation} \label{R_m=Q_m}
	{R}_{m}(f) 
	\ = \
	\sum_{s \in Z_*(m)} f(2^{-m-1}s)\varphi^*_{m,s},
\end{equation}
where for $k\in \NN_0$,
\begin{equation*}\label{phi^*_ks}
\varphi^*_{m,s}(x):=
\varphi(2^{m+1}x - s + 1), \quad s \in Z_*(m):=\{1,\ldots, 2^{m+1} - 1\}.
\end{equation*}
Hence one can see that ${R}_{m}(f)$ interpolates $f$ at the points 
$2^{-m-1}s$, $s \in Z_*(m)$, that is, 
\begin{equation*} \label{R_m=f}
	{R}_{m}(f)(2^{-m-1}s) 
	\ = \
	f(2^{-m-1}s), \quad 
	s \in Z_*(m). 
\end{equation*}

We  next extend the representation \eqref{eq:FaberRepresentationd=1} to functions in $C(\IId)$ by tensorization of the univariate  Faber basis.
Putting 
$$
Z(\bk):={\mathlarger{\mathlarger{\mathlarger{\mathlarger{\times}}}}}_{i=1}^d Z(k_i),
$$
for $\bk \in \NNdn$, $\bs \in Z(\bk)$, we introduce  the tensor product Faber functions
\begin{equation*} \label{hat-function}
\varphi_{\bk,\bs}(\bx)
\ := \
\prod_{i=1}^d \varphi_{k_i,s_i}(x_i),\quad \bx\in \IId,
\end{equation*}
and define the linear functionals $\lambda_{\bk,\bs}$ for multivariate function $f$ on $\IId$ by 
\begin{equation*} 
\lambda_{\bk,\bs}(f) \ := 
\prod_{i=1}^d \lambda_{k_i,s_i}(f),
\end{equation*}
where the univariate functional $\lambda_{k_i,s_i}$ is applied to the univariate function $f$ by considering $f$ as a function of variable $x_i$ with the other variables held fixed. 


\begin{lemma} \label{lemma[convergence(d)]}
	The  tensorized functions
$
	 \big\{\varphi_{\bk,\bs}: \ \bk \in \NNdn,\, \bs\in Z(\bk)\big\}
$ are a basis in $C(\IId)$. Moreover,  every function 
	$f \in C(\IId)$ can be represented by the tensorized Faber series 
	\begin{equation} \label{eq:FaberRepresentation}
	f
		\ = \
	\sum_{\bk \in \NNdn} q_{\bk}(f), \ \qquad q_{\bk}(f) := \sum_{\bs\in Z(\bk)} \lambda_{\bk,\bs}(f)\varphi_{\bk,\bs}
	\end{equation}
	converging in the norm of $C(\IId)$.
\end{lemma}
 The decomposition \eqref{eq:FaberRepresentation} when $d=2$ and an extension for function spaces with mixed smoothness  was obtained  independently in  \cite[Theorem 3.10]{Tri10B} and in \cite[Section 4]{Dung11a}. A generalization for the case $d\ge 2$ and also to B-spline interpolation and quasi-interpolation representation was established  in \cite{Dung11a,Dung16}.

When $f \in \Uas$, $\lambda_{\bk,\bs}(f)=0$ if $k_j=-1$ for some $j\in \{1,\ldots,d\}$, hence we can write 
\begin{equation*}
f \ = \ \sum_{\bk \in \NNd_0} q_\bk(f)   
\end{equation*}
with unconditional convergence in $C(\IId)$, see \cite[Theorem 3.13]{Tri10B}.
In this case it holds the following estimate
	\begin{equation}\label{eq:lambda-estimate}
	\begin{aligned}
|\lambda_{\bk,\bs}(f)| &=  2^{-d}\bigg|\prod_{i = 1}^d
 \Delta_{2^{-k_i-1}}^2 f\big(2^{-\bk}\bs\big) \bigg|
 \\
&= 2^{-d}\bigg|\prod_{i = 1}^d
		\Big[\Delta_{2^{-k_i-1}} f\big(2^{-\bk}\bs+2^{-k_i-1}\bee^i\big)-\Delta_{2^{-k_i-1}} f\big(2^{-\bk}\bs \big)\Big] \bigg|
		\leq 2^{ -\alpha d}2^{- \alpha|\bk|_1},
	\end{aligned}
\end{equation}
for $\bk\in \NNd_0, \ \bs\in Z(\bk)$. Here
$\{\bee^i\}_{i=1}^d$ is the standard basis of $\RR^d$.

For $f \in \mathring{C}(\IId)$,  we define  the  truncation of  Faber series  ${R}_m(f)$ by
\begin{equation} \label{R_m^d}
{R}_m(f) 
:= \ 
\sum_{\bk\in \NN_0^d,|\bk|_1 \leq m} q_\bk(f)
\ = \
\sum_{\bk\in \NN_0^d,|\bk|_1 \leq m} \ \sum_{s \in Z(\bk)} \lambda_{\bk,\bs}(f)\varphi_{\bk,\bs}.
\end{equation}
 The  function  ${R}_m(f)$ belongs to $ \mathring{C}(\IId)$ and is completely determined by sampled values  of $f$ at the points in the Smolyak grid
\[
G^d(m):= \big\{\bxi_{\bk,\bs} = 2^{-\bk-\bone}\bs:\, |\bk|_1 = m,\, \bs \in Z_*(\bk)  \big\},
\]
where $\bone=(1,\ldots,1)\in \NN^d$
and
 $$
 Z_*(\bk):= {\mathlarger{\mathlarger{\mathlarger{\mathlarger{\times}}}}}_{j=1}^d Z_*(k_j).
 $$ 
Moreover, ${R}_m(f)$ interpolates $f$ at the points 
$\bxi \in G^d(m)$.
\begin{equation*} \label{R^d_m=f}
{R}_m(f)(\bxi ) 
\ = \
f(\bxi), \quad 
\bxi  \in G^d(m).
\end{equation*}
Thus, the truncation of the Faber series ${R}_m(f)$ can be seen as a formula of interpolation sampling recovery on the grids  $G^d(m)$ for  $f \in  \mathring{C}(\IId)$.   

Notice that the Smolyak grids $G^d(m)$   are very sparse. The number of knots in $G^d(m)$ is smaller than $\frac{2^d}{(d-1)!}2^m m^{d-1}$ and is much smaller than $2^{dm}$, the number of knots in  corresponding standard full grids. However, for periodic functions having mixed smoothness,  they give the same error of the sampling recovery on the standard full grids.  See \cite[Chapter 5]{DTU18B} for details.

The following lemma gives a $d$-dependent estimate of the  error of approximation by the sparse-grid interpolation operators $R_m(f)$ of functions having mixed smoothness from $ \Uas$.

\begin{lemma} \label{thm-DT20}
	Let  $d \ge 2$, $m\in \NN$, and $0 < \alpha \le 1$. Then we have
\begin{equation*}  
		\begin{aligned}
			\sup_{f \in \Uas} \|f - {R}_m(f)\|_\infty
			&\le  
			2^{-\alpha} B^{d}\, 2^{- \alpha m} \, \binom{m+d}{d-1},\qquad B=(2^\alpha-1)^{-1}.
		\end{aligned}
	\end{equation*}
\end{lemma}

\begin{proof}
	For every  $f \in \Uas$ and $\bk \in \NNd_0$, as  the functions $\varphi_{\bk,\bs}$, $\bs\in Z(\bk)$, have disjoint supports, by \eqref{eq:lambda-estimate} we have
	\begin{equation*}  
	\begin{aligned}
		\|f - {R}_m(f)\|_\infty
		&\le  
		\sum_{\bk \in \NNd_0: \, |\bk|_1 > m}\Bigg\|\sum_{\bs\in Z(\bk)} \lambda_{\bk,\bs}(f)\varphi_{\bk,\bs} \Bigg\|_\infty
		\le  
		\sum_{\bk \in \NNd_0: \, |\bk|_1 > m}
		2^{ -\alpha d}2^{- \alpha|\bk|_1}
		\\[1.5ex]
		&=  
		2^{ -\alpha d} \sum_{\ell=m+1}^\infty
		\binom{\ell+d-1}{d-1}\, 2^{- \alpha \ell}  =  
		2^{ -\alpha d} \sum_{s=0}^\infty
		\binom{m+s+d}{d-1}\, 2^{- \alpha (s+m+1)}  \\[1.5ex]
		&=  
		2^{ -\alpha d}  2^{- \alpha (m+1)}
		\sum_{s=0}^\infty
		\binom{m+s+d}{d-1}\, 2^{- \alpha s}.
	\end{aligned}
\end{equation*}
Using
\begin{equation*}\label{eq-DT}
\sum_{s=0}^\infty\binom{m+s}{n}t^s \leq (1-t)^{-n-1}\binom{m}{n},\qquad t\in(0,1),
\end{equation*} 
see \cite[Lemma 2.2]{DuTh20}, we finally obtain
\begin{equation} \nonumber
	\begin{aligned}
		\|f - {R}_m(f)\|_\infty
		&\le 
		2^{ - \alpha  d}  2^{- \alpha (m+1)} (1-2^{-\alpha})^{-d} \binom{m+d}{d-1}
		= 2^{-\alpha} B^{d}\, 2^{- \alpha m} \, \binom{m+d}{d-1}.
	\end{aligned}
\end{equation}
\hfill	
\end{proof}

\subsection{Approximation by sets of finite cardinality}

In this subsection, we explicitly construct a set of finite cardinality for approximation of $f \in \Uas$ and give an estimate of the approximation error as well as the cardinality of this  set.

Again, we start with the univariate case. For  $f \in \Uasone$ we explicitly construct the function $S_f \in \Uasone$ by
\begin{equation} \label{approx:S_f}
S_f := \ \sum_{s \in Z_*(m)} 2^{-\alpha (m+1)}l_s(f) \varphi^*_{m,s},
\end{equation}
where we put $l_0(f)=0$  and assign the values $ S_f(2^{-m-1}s) = 2^{-\alpha (m+1)}l_s (f)$ from left to right closest to $f(2^{-m-1}s)$ for $s = 1,\ldots, 2^{m+1}-1$. If there are two possible choices for $l_s(f)$ we choose $l_s(f)$ 
that is closest to the already determined $l_{s-1}(f)$.
We define 
\begin{equation*} \label{Ss^{alpha}(m)}
\Ss^{\alpha}(m) := \big\{S_f: f\in \Uasone \big\}.
\end{equation*}
\begin{lemma}\label{lem:pattern,d=1}
	Let $0<\alpha \leq 1$, $m \in \NN_0$. 
	Then it holds  $| \Ss^{\alpha}(m)|\leq 3^{2^{m+1}}$  and for  every $f \in \Uasone$ we have
	\[
	\|{R}_m(f) - S_f\|_{{L_\infty(\II)}} \leq 2^{-\alpha (m+1) -1}.
	\]  
\end{lemma}
\begin{proof} In this proof we develop a technique used in \cite{DDF.19}.   
	For  every  $f \in \Uasone$, from the construction of $S_f$ we have 
	\begin{equation*}
	S_f  = \ \sum_{s \in Z_*(m)} S_f(2^{-m-1}s) \varphi^*_{m,s}
	\end{equation*}
	and
	\begin{equation} \label{S-f}
	|S_f(2^{-m-1}s) - f(2^{-m-1}s)|  \le 2^{- \alpha (m+1) - 1}, \ \  s = 0,\ldots, 2^{m+1}.
	\end{equation}	
	From this, \eqref{R_m=Q_m} and the inequality $\sum_{s \in Z_*(m)} \varphi^*_{m,s}(x) \le 1$, we deduce that for every $x \in \II$,
	\begin{equation} \nonumber
	\begin{split}
	|{R}_{m}(f)(x) - S_f(x)|
	\ &\le \
	\sum_{s \in Z_*(m)} |f(2^{-m-1}s) - S_f(2^{-m-1}s)|\varphi^*_{m,s}(x)
	\\[1ex]
	\ &\le \
	2^{- \alpha (m+1) - 1}\sum_{s \in Z_*(m)} \varphi^*_{m,s}(x)
	\ \le \ 2^{- \alpha (m+1) - 1}.
	\end{split}
	\end{equation} 	
	We have by \eqref{S-f} and the inclusion  $f \in \Uasone$,
	\begin{equation*}
	\begin{aligned}
	|l_s(f) - l_{s-1}(f)|2^{- \alpha (m+1)}
	&=
	|S_f(2^{-m-1}s) - S_f(2^{-m-1}(s-1))| 
	\\[1.5ex]
	& \le
	|S_f(2^{-m-1}s) - f(2^{-m-1}s)|  +  |f(2^{-m-1}s) - f(2^{-m-1}(s-1))|  
	\\[1.5ex]
	& 
	\quad  +  |f(2^{-m-1}(s-1)) -  S(2^{-m-1}(s-1))|
	\le 2^{- \alpha (m+1) + 1}.
	\end{aligned}
	\end{equation*}
	Hence, $|l_s (f)- l_{s-1}(f)| \le 2$. But the case $|l_s (f)- l_{s-1}(f)| = 2$ is not possible since it would imply that $l_s(f)$ is not closest to $l_{s-1}(f)$. 
	This means that $|l_s(f)- l_{s-1}(f)| \le 1$.  Taking account this inequality and $l_0(f) = 0$, we can see that  $|\Ss^\alpha(m)| \le 3^{2^{m+1}}$.
	\hfill
\end{proof}

In the following, we make use the abbreviations:
$\bx_j := (x_1,\ldots,x_j) \in \RR^j$;  $\bar{\bx}_j := (x_{j+1},\ldots,x_d) \in \RR^{d-j}$ with the convention $\bx_0 := 0$ for 
$\bx \in \RRd$ and $j = 0,1,\ldots, d-1$.
When $j=1$ we denote $x_1$ instead of $\bx_1$. 

We now construct a set of finite cardinality for approximation of $f \in \Uas$. Our strategy is to apply the above result to explicitly  construct a set of finite cardinality for approximation of $R_m(f)$, and show that this set approximates $f$ as well as $R_m(f)$. To do this we need a special representation of ${R}_m(f)$ in terms of the tensor product of $\varphi_{\bar{\bk}_1,\bar{\bs}_1}(\bar{\bx}_1)$ and ${R}_{m-|\bar{\bk}_1|_1}$ of a function in  $\Uasone$ of variable $x_1$.
\begin{lemma}\label{lem:representation}
	Let  $d \ge 2$, $0<\alpha \leq 1$,  $m>1$ and $f \in \Uas$.  It holds the representation
	\begin{equation} \label{R_m=}
		{R}_m(f) (\bx) = \sum_{|\bar{\bk}_1|_1\leq m } \ \sum_{\bar{\bs}_1\in Z(\bar{\bk}_1)}  
		2^{-\alpha (|\bar{\bk}_1|_1 + d - 1)} \varphi_{\bar{\bk}_1,\bar{\bs}_1}(\bar{\bx}_1){R}_{m-|\bar{\bk}_1|_1}\big(K_{\bar{\bk}_1,\bar{\bs}_1}(f)(x_1)\big), 
	\end{equation}
	where the univariate function $K_{\bar{\bk}_1,\bar{\bs}_1}(f)$ belongs  to  $\Uasone$ and is defined by
	\begin{equation} \label{f_k^j,s^j}
		K_{\bar{\bk}_1,\bar{\bs}_1}(f)(x_1):=\prod_{j =2}^d \bigg(- \frac {1}{2} 2^{\alpha(k_j+1)} \Delta_{2^{-k_j-1}}^2 f\big(x_1,2^{-\bar{\bk}_1}\bar{\bs}_1\big)\bigg).
	\end{equation}	
\end{lemma}
\begin{proof}
	We have that
	\begin{equation*}
		\begin{aligned}
			{R}_m(f)(\bx)
			&=
			\sum_{k_2=0}^m \sum_{k_3=0}^{m-k_2} \ldots \sum_{k_{d}=0}^{m-k_2-\ldots-k_{d-1}}\Bigg[\sum_{k_1=0}^{m-|\bar{\bk}_1|_1} q_{k_1}\Bigg( \prod_{j=2}^{d}q_{k_j}(f)\Bigg)(\bx)\Bigg]
			\\
			&=
			\sum_{k_2=0}^m \sum_{k_3=0}^{m-k_2} \ldots \sum_{k_{d}=0}^{m-k_2-\ldots-k_{d-1}} {R}_{m-|\bar{\bk}_1|_1}\Bigg( \prod_{j=2}^{d}q_{k_j}(f)\Bigg)(\bx)
			\\
			&=
			\sum_{|\bar{\bk}_1|_1\leq m} {R}_{m-|\bar{\bk}_1|_1}\Bigg(\sum_{\bar{\bs}_1 \in Z(\bar{\bk}_1)}  \prod_{j=2}^{d}\bigg(- \frac {1}{2} \Delta_{2^{-k_j-1}}^2 f\big(x_1,2^{-\bar{\bk}_1}\bar{\bs}_1\big)\bigg)\varphi_{\bar{\bk}_1,\bar{\bs}_1}(\bar{\bx}_1)\Bigg).
		\end{aligned}
	\end{equation*}
	Here ${R}_{m-|\bar{\bk}_1|_1}$ applies to the function of variable $x_1$.	Hence, we can write
	\begin{equation*}
		{R}_m(f) (\bx) = \sum_{|\bar{\bk}_1|_1\leq m } \sum_{\bar{\bs}_1\in Z(\bar{\bk}_1)}  \Bigg(\prod_{j=2}^{d}2^{-\alpha(k_j+1)} \Bigg)\varphi_{\bar{\bk}_1,\bar{\bs}_1}(\bar{\bx}_1){R}_{m-|\bar{\bk}_1|_1}\big(K_{\bar{\bk}_1,\bar{\bs}_1}(f)(x_1)\big).
	\end{equation*}	
	Thus, \eqref{R_m=} is proven.  For $x_1\in \II$ and $x_1 + h_1\in \II$, is holds the estimates
	\[
	|\Delta_{h_1} K_{\bar{\bk}_1,\bar{\bs}_1}(f)(x_1)| \leq \Bigg( \prod_{j=2}^{d}  \frac{1}{2}\cdot 2^{\alpha(k_j+1)} \cdot 2\cdot 2^{-\alpha(k_j+1)} \Bigg) |h_1|^\alpha \leq h_1^\alpha
	\]
	which implies that $K_{\bar{\bk}_1,\bar{\bs}_1}(f) \in \Uasone$.
	\hfill
\end{proof}

From the above special representation of $R_m(f)$ and   Lemmata \ref{thm-DT20} and  \ref{lem:pattern,d=1} we derive the following result.
\begin{lemma}\label{lem:pattern}
	Let   $m>1$, $d \ge 2$ and $0<\alpha \leq 1$.  For $f \in \Uas$, let the function ${S}_m(f)$ be defined by
\begin{equation} \label{S_f}
{S}_m(f)(\bx):= \sum_{|\bar{\bk}_1|\leq m }  2^{-\alpha (|\bar{\bk}_1|_1 + d - 1)}
\ \sum_{\bar{\bs}_1\in Z(\bar{\bk}_1)} 
\varphi_{\bar{\bk}_1,\bar{\bs}_1}(\bar{\bx}_1) S_{K_{\bar{\bk}_1,\bar{\bs}_1}(f)}(x_1),
\end{equation}
 where $S_{K_{\bar{\bk}_1,\bar{\bs}_1}(f)}\in \Ss^\alpha(m - |\bar{\bk}_1|_1)$  is  as in \eqref{approx:S_f} for the function $K_{\bar{\bk}_1,\bar{\bs}_1}(f)$. Then  it holds the inequality 	
\begin{equation} \label{f-S<}
\|f-{S}_m(f)\|_\infty \leq   B^{d}  2^{-\alpha m} \binom{m+d}{d-1}.
\end{equation}
 Moreover, for the set
\begin{equation*} \label{Ss^{d,alpha}(m)}
\Ss^{\alpha,d}(m):=\big\{{S}_m(f): \ f\in \Uas \big\},
\end{equation*}
we have 
 $N_{d}(m):=|\Ss^{\alpha,d}(m)| \le 3^{2^{m+1}\binom{m+d-1}{d-1}}$.
\end{lemma}

\begin{proof}  We first estimate  $N_{d}(m)$.
	Since the number of $\bar{\bk}_1$ with $|\bar{\bk}_1|_1\leq m$ is $\binom{m+d-1}{d-1}$,
and the cardinality of $\Ss^\alpha(m - |\bar{\bk}_1|_1)$ is  bounded by $ 3^{2^{m-|\bar{\bk}_1|_1 + 1}}$,
we have 
\[
 N_{d}(m) \leq \bigg( \Big(  3^{2^{m-|\bar{\bk}_1|_1 + 1}}\Big)^{2^{|\bar{\bk}_1|_1}} \bigg)^{\binom{m+d-1}{d-1}} =  3^{2^{m+1}\binom{m+d-1}{d-1}}.
\]
 Lemma \ref{thm-DT20}  gives 
	\begin{equation}\label{eq-f-Rmf}
	\|f-{R}_m(f) \|_\infty \leq 2^{-\alpha}  B^{d}  2^{-\alpha m} \binom{m+d}{d-1}.
	\end{equation}	
Next we show that 
	\begin{equation} \label{R-S}
	\|{R}_m(f) - {S}_m(f)\|_\infty \leq 2^{-1-\alpha d}  2^{-\alpha m} \binom{m+d}{d-1}.
	\end{equation}	 
Since $K_{\bar{\bk}_1,\bar{\bs}_1}(f) \in \Uasone$ by Lemma \ref{lem:representation}, applying  Lemma \ref{lem:pattern,d=1} we deduce that
	there  is $S_{K_{\bar{\bk}_1,\bar{\bs}_1}(f)} \in \Ss^\alpha\big(m - |\bar{\bk}_1|_1\big)$ such that
	\begin{equation} \label{R-S(2)}
	\big\|{R}_{m-|\bar{\bk}_1|_1}\big(K_{\bar{\bk}_1,\bar{\bs}_1}(f)\big) - S_{K_{\bar{\bk}_1,\bar{\bs}_1}(f)}  \big\|_\infty \leq 2^{-1-\alpha} 2^{-\alpha (m-|\bar{\bk}_1|_1)}.
	\end{equation}
     Hence, taking account that the supports of $\varphi_{\bar{\bk}_1,\bar{\bs}_1}$ with $\bar{\bs}_1 \in Z(\bar{\bk}_1)$ are disjoint,  by the  representation \eqref{R_m=}--\eqref{f_k^j,s^j} of $R_m(f)$ and \eqref{R-S(2)} we have that  for every $\bx \in \IId$, 
	\begin{equation*}
	\begin{aligned}
&	\big|{R}_m(f)(\bx) - {S}_m(f)(\bx)\big| 
\\
& = \Bigg | \sum_{|\bar{\bk}_1|\leq m }\sum_{\bar{\bs}_1 \in Z(\bar{\bk}_1) } 2^{-\alpha (|\bar{\bk}_1|_1 + d - 1)} 
	\varphi_{\bar{\bk}_1,\bar{\bs}_1}(\bar{\bx}_1) \Big( {R}_{m-|\bar{\bk}_1|_1}\big(K_{\bar{\bk}_1,\bar{\bs}_1}(f)(x_1)\big) - S_{K_{\bar{\bk}_1,\bar{\bs}_1}(f)} (x_1) \Big)\Bigg|
	\\
	& \leq  \sum_{|\bar{\bk}_1|\leq m } 2^{-\alpha (|\bar{\bk}_1|_1 + d - 1)}  \sup_{\bar{\bs}_1 \in Z(\bar{\bk}_1) }\Big\|{R}_{m-|\bar{\bk}_1|_1}\big(K_{\bar{\bk}_1,\bar{\bs}_1}(f) \big) - S_{K_{\bar{\bk}_1,\bar{\bs}_1}(f)}   \Big\|_\infty
	\\
	&\leq  2^{-1-\alpha} 2^{-\alpha(d-1)}  \sum_{|\bar{\bk}_1|\leq m } 2^{-\alpha|\bar{\bk}_1|_1} 2^{-\alpha(m-|\bar{\bk}_1|_1)}
	= 2^{-1-\alpha d}  2^{-\alpha m} \binom{m+d-1}{d-1}
	\end{aligned}
	\end{equation*}
This implies \eqref{R-S}.  From \eqref{eq-f-Rmf} and \eqref{R-S}, the triangle inequality and $
2^{-\alpha}B^{d} + 2^{-1-\alpha d} \leq B^{d}
$ prove \eqref{f-S<}.
	\hfill
\end{proof}

 \section{ Nonlinear approximation by parametric manifolds}\label{sec-nonlinear-manifold}
  This section  aims at constructing  nonlinear methods of parametric manifold approximation of functions $f\in \Uas$. More precisely, we construct such a nonlinear method $Q_N(f) = G^*_N((\blambda^*_N(f)))$ with mappings $\blambda^*_N$ and $G^*_N$ of the form \eqref{lambda_N,G_N}, satisfying the upper and lower estimates of approximation error \eqref{upperbnd}--\eqref{lowerbnd}. In order to do this we use   the truncation of the tensorized Faber series ${R}_n(f)$ as an  intermediate approximation.  We then represent    the difference $f- {R}_n(f)$ in a special form and approximate terms in this  representation by functions in the set of finite cardinality constructed in the previous section. 
 
 For univariate functions $f \in \mathring{C}(\II)$, let the operator ${T}_k$, $k \in \NN_0$, be defined by 
 \[
 {T}_k(f):= f - {R}_{k-1}(f)
 \]
 with the operator ${R}_k$ defined as in \eqref{R_m^d} and  the convention ${R}_{-1}:=0$. From this definition we have ${T}_0$ is the identity operator.  Notice that for $f \in \Uasone$, it holds the inequality $\|{T}_k(f)\|_{H^{\alpha}_{\infty}(\II)} \le 2$.
 For a multivariate function $f \in \mathring{C}(\IId)$, the tensor product operator  ${T}_\bk$, $\bk=(k_1,\ldots,k_d) \in \NN_0^d$, is defined by
 \[
 {T}_{\bk}(f) := \prod_{j=1}^d {T}_{k_j}(f),
 \]
 where the univariate operator ${T}_{k_j}$ is applied to the univariate function $f$ by considering $f$ as a function of variable $x_j$ with the other variables held fixed. It holds that $\|{T}_\bk(f)\|_{H^{\alpha}_{\infty}(\IId)} \le 2^d$.

For $n\in \NN$ we have
 \begin{equation*}
 \begin{aligned}
 f- {R}_n(f)
 & = \sum_{\bk\in \NN_0^d, |\bk|_1>n}q_{\bk}(f)
 = \sum_{k_1>n\atop k_j\geq 0, j=2,\ldots,d} q_{\bk}(f)  + \sum_{k_1=0}^n q_{k_1}\Bigg( \sum_{|\bar{\bk}_1|_1>n-k_1}q_{\bar{\bk}_1}(f)\Bigg)
 \\
 &= {T}_{(n+1)\bee^1}(f) + \sum_{k_1=0}^n q_{k_1}\Bigg( {T}_{(n+1-k_1)\bee^2}(f)+ \sum_{k_2=0}^{n-k_1}q_{k_2}\bigg(\sum_{|\bar{\bk}_2|_1>n-|\bk_2|_1}q_{\bar{\bk}_2}(f)\bigg)\Bigg)
 \\
 &= {T}_{(n+1)\bee^1}(f) + \sum_{k_1=0}^n q_{k_1}  {T}_{(n+1-k_1)\bee^2}(f)+ \sum_{|\bk_2|_1\leq n}q_{\bk_2}\Bigg(\sum_{|\bar{\bk}_2|_1>n-|\bk_2|_1}q_{\bar{\bk}_2}(f)\Bigg).
 \end{aligned}
 \end{equation*} 
 Continuing  in this way, we arrive at
 \begin{equation*}
 \begin{aligned}
 f-{R}_n(f)& = {T}_{(n+1)\bee^1}(f) + \sum_{k_1=0}^n q_{k_1}  {T}_{(n+1-k_1)\bee^2}(f) + \ldots 
 +  \sum_{|\bk_{d-1}|_1\leq n}q_{\bk_{d-1}}\Big({T}_{(n+1-|\bk_{d-1}|_1)\bee^d}(f)\Big)
 \\
 & = {T}_{(n+1)\bee^1}(f) + \sum_{k_1=0}^n   {T}_{(n+1-k_1)\bee^2}\big(q_{k_1}(f)\big) + \ldots 
 +  \sum_{|\bk_{d-1}|_1\leq n}{T}_{(n+1-|\bk_{d-1}|_1)\bee^d}\big(q_{\bk_{d-1}}(f)\big).
 \end{aligned}
 \end{equation*}
Putting 
 $$
F_{\bk_j}
:= 
{T}_{(n+1-|\bk_j|_1)\bee^{j+1}}\big(q_{\bk_j}(f)\big),\ \ j = 0,1,\ldots,d-1.
$$
we can write 
 \begin{equation}\label{eq:representation}
 \begin{aligned}
 f-{R}_n(f)& 
 = \sum_{j = 0}^{d-1} \ \sum_{|\bk_j|_1 \leq n}F_{\bk_j},
 \end{aligned}
 \end{equation}

Let $f\in \Uas$ be given.  We will use this special representation to explicitly construct mappings $\blambda^*_N$ and $G^*_N$ of the form \eqref{lambda_N,G_N}. To this end, caused by \eqref{eq:representation}, we will preliminarly  approximate  $T_\bk(f)$.
 
Put 
 \[
 I_{\bk,\bs}
 := 
 {\mathlarger{\mathlarger{\mathlarger{\mathlarger{\times}}}}}_{j=1}^d I_{k_j, s_j} 
 = 
 {\mathlarger{\mathlarger{\mathlarger{\mathlarger{\times}}}}}_{j=1}^d  [2^{-k_j}s_j,2^{-k_j}(s_j+1)], \ \ \bk \in \NN_0^d,\ \ \bs \in Z(\bk),
 \]
and
 $$
 {T}_{\bk,\bs}(f)(\bx):=  2^{\alpha|\bk|_1 - d}\big({T}_{\bk}(f)\chi_{I_{\bk,\bs}}\big)\big( 2^{-\bk}(\bx+\bs)\big).
 $$ 
 Since $\supp \big({T}_{\bk}(f)\chi_{I_{\bk,\bs}}\big) \subset I_{\bk,\bs}$ and $\|{T}_{\bk}(f)\chi_{I_{\bk,\bs}}\|_{\Lad}\leq 2^d$, we have that
\begin{equation} \label{T_k,s}
 	\supp \big({T}_{\bk,\bs}(f) \big)\subset \IId,\qquad  {T}_{\bk,\bs}(f) \in \Uas.
 \end{equation}
 This allows us to apply  Lemma \ref{lem:pattern} to the functions ${T}_{\bk,\bs}(f)$. Namely, according to this lemma we explicitly construct the function ${S}_m({{T}_{\bk,\bs}(f)})$ by the formula \eqref{S_f} so that we have by \eqref{f-S<}
   \begin{equation}\label{T_ks-S_m}
 \big\|{T}_{\bk,\bs}(f)-{S}_m({{T}_{\bk,\bs}(f)})\big\|_\infty 
 \leq 
 B^{d}  2^{-\alpha m} \binom{m+d}{d-1}.
 \end{equation}
Define 
  \begin{equation}\label{eq-S-Tkf}
{S}_{\bk,m}(f)(\bx):= 2^{-\alpha |\bk|_1 + d}\sum_{\bs\in Z(\bk)} {S}_m\big({{T}_{\bk,\bs}(f)}\big)\big(2^\bk\bx-\bs\big).
\end{equation}
 We then get
 \begin{equation*}
\begin{aligned}
 \big\| {T}_{\bk}(f) - {S}_{\bk,m}(f) \big\|_\infty
 & =\Bigg \| \sum_{\bs \in Z(\bk)} \Big[
{T}_{\bk}(f)\chi_{I_{\bk,\bs}}(\cdot) -2^{-\alpha |\bk|_1 + d}  {S}_m\big({{T}_{\bk,\bs}(f)}\big)\big(2^\bk\cdot-\bs\big)\Big] \Bigg\|_\infty  \\
&= 2^{-\alpha |\bk|_1 + d} \Bigg \| \sum_{\bs \in Z(\bk)} \Big[
{T}_{\bk,\bs} (f) -  {S}_m\big({{T}_{\bk,\bs}(f)}\big)\Big]\big(2^\bk \cdot-\bs\big) \Bigg\|_\infty, 
\end{aligned}
 \end{equation*}
and, consequently, by the equality $|Z(\bk)| = 2^{|\bk|_1}$ and \eqref{T_ks-S_m}, the estimate of the error of the approximation $ {T}_{\bk}(f)$ by ${S}_{\bk,m}(f)$
 \begin{equation}\label{eq:Tkf}
\begin{aligned}
\big\| {T}_{\bk}(f) - {S}_{\bk,m}(f) \big\|_\infty
&\leq   (2B)^{d} \big(2^{m} 2^{|\bk|_1}\big)^{-\alpha}\binom{m+d}{d-1}.
\end{aligned}
\end{equation}

 Let $\Ff^d(m)$ be the  finite-dimensional subspace in  $\mathring{C}(\IId)$ of the form
 \begin{equation} \label{Ff^d(m)}
 g=\sum_{\bk \in \NN_0^d, |\bk|_1 \leq m} \sum_{\bs\in Z(\bk)} \alpha_{\bk,\bs} \varphi_{\bk,\bs},\qquad \alpha_{\bk,\bs}\in \RR.
 \end{equation}
 It is easy to see that $R_m(f) \in \Ff^d(m) $ for $f \in \mathring{C}(\IId)$ and $\dim \Ff^d(m) = \sum_{\ell=0}^m 2^\ell\binom{\ell+d-1}{d-1}$.
 
	In the following, for any $N\in \NN, \ N\geq N_0$ we will  explicitly construct the maps
	$$
	\blambda^*_N: \Uas \to    \RR^N \ \ {\text and}  \ \ G^*_N: \RR^N \to \Ff^d\big(\lfloor \log N\rfloor + \lfloor \log\log N\rfloor +1 \big)
	$$ 
	and estimate the approximation error
	$\sup_{f\in \Uas}\|f-G_N^*(\blambda^*_N(f))\|_\infty$ in terms of $N$.

For $j=0, 1,\ldots, d-1$ we put
\begin{equation*}\label{eq-Mdm}
M_{d-j}(m):= N_{d-j}(m)\dim\big(\Ff^{d-j}(m)\big),
\end{equation*}
where recall, $N_{d-j}(m):=|\Ss^{\alpha,d-j}(m)|$, see Lemma \ref{lem:pattern}. We have
\begin{equation}\label{eq-Mdj}
M_{d-j}(m) \leq 3^{2^{m+1}\binom{m+d-j-1}{d-j-1}}\sum_{\ell=0}^m 2^\ell \binom{\ell+d-j-1}{d-j-1}\leq 3^{2^{m+1}\binom{m+d-j-1}{d-j-1}} 2^{m+1}\binom{m+d-j-1}{d-j-1} .
\end{equation}  
Let  
$\Gamma_j(n)$ be   the set of all triples $(\bk_j, \bs_j, s_{j+1})$ satisfying the condition 
\[|\bk_j|_1 \le n, \quad \bs_j \in Z(\bk_j), \quad s_{j+1} = 0,\ldots, 2^{n+1-|\bk_j|_1} - 1,
\]
in particular, $\Gamma_0(n)=\{s_1: 0\leq s_1\leq 2^{n+1} - 1 \}$. We have
\begin{equation}\label{eq-Gammajn}
|\Gamma_j(n)|= \sum_{|\bk_j|\leq n} 2^{|\bk_j|_1}\cdot 2^{n+1-|\bk_j|_1}=2^{n+1}\sum_{k=0}^n\binom{k+j-1}{j-1}= 2^{n+1}\binom{n+j}{j}
\end{equation}
for all $j=0,\ldots,d-1$. 

For $\eta \in [N_{d-j}(m)]:=\{1,\ldots,N_{d-j}(m)\}$ and a sequence 
$$\ba^\eta =\big(a_{\bar{\bell}_j,\bar{\bt}_j}^{\eta}\big)_{ |\bar{\bell}_j|_1\leq m, \, \bar{\bt}_j \in Z(\bar{\bell}_j)}  \in \RR^{\dim (\Ff^{d-j}(m))},$$
we put 
$$
S_{\ba^\eta}(\bar{\bx}_j):= \sum_{ |\bar{\bell}_j|_1\leq m}\sum_{\bar{\bt}_j \in Z(\bar{\bell}_j)} a_{\bar{\bell}_j,\bar{\bt}_j}^{\eta}\prod_{i=j+1}^d\varphi_{\ell_i,t_i}(x_i).
$$
If
$$
\ba: =(\ba^\eta)_{\eta = 1}^{N_{d-j}(m)} \in \RR^{M_{d-j}(m)}, 
$$
and
$$\btheta: =(\theta_{(\bk_j,\bs_j,s_{j+1})})_{(\bk_j,\bs_j,s_{j+1})\in \Gamma_j(n)} \in [N_{d-j}(m)]^{|\Gamma_j(n)|}, $$
that is, elements of $\btheta$ are numbered by indices $(\bk_j,\bs_j,s_{j+1})\in \Gamma_j(n)$, we define the maps
\begin{equation} \label{G^j_{m,n}}
G^j_{m,n}: \RR^{M_{d-j}(m)} \times [N_{d-j}(m)]^{|\Gamma_j(n)|} \ \to \ \Ff^d(m+n+1), \ j=0,\ldots,d-1,
\end{equation}
by  $$
\blambda^j := (\ba,\, \btheta) \   \mapsto \ G^j_{m,n}(\blambda^j) ,$$
where
\begin{equation*} \label{eq:form}
\begin{aligned}
G^j_{m,n}(\blambda^j) : & = \sum_{\eta=1}^{N_{d-j}(m)} \sum_{(\bk_j,\bs_j,s_{j+1}): \atop \theta_{(\bk_j,\bs_j,s_{j+1})}=\eta} 	
\frac{2^{ d-j}}{2^{\alpha (n+1+j)}} \varphi_{\bk_j,\bs_j}(\bx_j) 
S_{\ba^\eta}
\big(2^{n+1-|\bk_j|_1}x_{j+1} - s_{j+1},\bar{\bx}_{j+1}\big)
\end{aligned}
\end{equation*}
if $j=1,\ldots,d-1,$  and
\begin{equation*} \label{eq:form-j=0}
\begin{aligned}
G^0_{m,n}(\blambda^0) : & = \sum_{\eta=1}^{N_{d}(m)} \sum_{s_1: \theta_{s_1}=\eta }  
2^{-\alpha (n+1)  + d}
S_{\ba^\eta}
\big(2^{n+1}x_{1} - s_{1},\bar{\bx}_{1}\big).
\end{aligned}
\end{equation*}
 These maps are well-defined in the sense that   the functions $G^j_{m,n}( \blambda^j)$, $j=0,\ldots,d-1$,  belong to $\Ff^d(m+n+1)$. We prove this for the case $j=1,\ldots,d-1$. The case  $j=0$ can be carried out similarly. For the function 
	$\varphi_{\bar{\bell}_j,\bar{\bt}_j} (\bar{\bx}_j)$ with $|\bar{\bell}_j|_1\leq m, \bar{\bt}_j \in Z(\bar{\bell}_j)$, we have that
\begin{align*}
&\varphi_{\bar{\bell}_j,\bar{\bt}_j}
\big(2^{n+1-|\bk_j|_1}x_{j+1} - s_{j+1},\bar{\bx}_{j+1}\big)
= \bigg(\prod_{i=j+2}^d\varphi_{\ell_i,t_i}(x_i)\bigg)\varphi_{\ell_{j+1},t_{j+1}}\big(2^{n+1-|\bk_j|_1}x_{j+1} - s_{j+1}\big)
\\
&= \bigg(\prod_{i=j+2}^d\varphi_{\ell_i,t_i}(x_i)\bigg)\varphi\big(2^{\ell_{j+1}+n+2-|\bk_j|_1}x_{j+1} - 2^{\ell_{j+1}+1}s_{j+1} -2t_{j+1}\big)
\\
& =\bigg(\prod_{i=j+2}^d\varphi_{\ell_i,t_i}(x_i)\bigg)\varphi_{\ell_{j+1}+n+1-|\bk_j|_1,2^{\ell_{j+1}}s_{j+1} +t_{j+1}}(x_{j+1}) \in \Ff^{d-j}(n+m+1-|\bk_j|_1).
\end{align*}
Since $S_{\ba^\eta}$ is a linear combination of $\varphi_{\bar{\bell}_j,\bar{\bt}_j}$ with $|\bar{\bell}_j|_1\leq m, \bar{\bt}_j \in Z(\bar{\bell}_j)$,  we conclude that
\begin{align*}
S_{\ba^\eta}
\big(2^{n+1-|\bk_j|_1}x_{j+1} - s_{j+1},\bar{\bx}_{j+1}\big) \in \Ff^{d-j}(n+m+1-|\bk_j|_1)
\end{align*}
which implies 
$$
\varphi_{\bk_j,\bs_j}(\bx_j)  
S_{\ba^\eta}
\big(2^{n+1-|\bk_j|_1}x_{j+1} - s_{j+1},\bar{\bx}_{j+1}\big) \in \Ff^d(m+n+1).
$$

In the following lemma we explicitly construct a preliminary approximation of $f- {{R}_{n}}(f)$ and estimate the approximation error.

\begin{lemma}\label{prop-f-Rn-Gnm} 
 Let $\alpha \in (0,1]$,  $j=0,\ldots,d-1$, and $m,n\in \NN$.  Then we can explicitly construct a map
\begin{equation} \label{lambda^j_{m,n}}
	\blambda^j_{m,n}:\ \Uas \to    \RR^{M_{d-j}(m)} \times [N_{d-j}(m)]^{|\Gamma_j(n)|}
\end{equation}
	 so that  for every $f\in \Uas$,
\begin{equation*} \label{eq:main-estimate}
\bigg\| f- {{R}_{n}}(f)-\sum_{j=0}^{d-1}G^j_{m,n}(\blambda^j_{m,n}(f))\bigg\|_\infty
\leq 
2^{-\alpha+1}(2B)^{d}  2^{-\alpha(m+n)}\binom{m+n+d}{d-1},
\end{equation*}
where $B$ is given in Lemma \ref{thm-DT20}.
\end{lemma}
\begin{proof} {\it Step 1.}   
	We auxiliarily  construct an approximation and estimate the approximation error for all  $F_{\bk_j}$, $j = 0,1,\ldots,d-1$.
	For $F_{\bk_0} = {T}_{(n+1)\bee^1}(f) $ we take ${S}_{(n+1)\bee^1,m}(f)$ by formula \eqref{eq-S-Tkf} and apply \eqref{eq:Tkf} to obtain the estimate
	\begin{equation}\label{eq-F-k0}
	\begin{aligned}
	\big\|F_{\bk_0} - {S}_{(n+1)\bee^1,m}(f)\big\|_\infty 
	&\leq   
   (2B)^{d} \big(2^m 2^{|(n+1)\bee^1|_1}\big)^{-\alpha}\binom{m+d}{d-1} 
	\\
	&= 2^{-\alpha} (2B)^{d}  2^{-\alpha(n+m)}\binom{m+d}{d-1}.
	\end{aligned}
	\end{equation}
	For $j = 1, \ldots, d-1$, we rewrite $F_{\bk_j}(\bx)$ in the form
	\begin{equation*}
	\begin{split}
	F_{\bk_j}(\bx) 
	&=  {T}_{(n+1-|\bk_j|_1)\bee^{j+1}}  \Bigg(\sum_{\bs_j \in Z(\bk_j)}  	\bigg((-1)^j2^{-j}\prod_{i=0}^j  \Delta_{2^{-k_i-1}}^2
	f\big(2^{-\bk_j}\bs_j,\bar{\bx}_j\big)\bigg)\varphi_{\bk_j,\bs_j}(\bx_j) \Bigg)
	\\
	&= 
	\sum_{\bs_j \in Z(\bk_j)} \varphi_{\bk_j,\bs_j}(\bx_j) 
	\bigg( {T}_{\bar{\bk}_j^*} 
	\bigg((-1)^j2^{-j}\prod_{i=0}^j  \Delta_{2^{-k_i-1}}^2
	f\big(2^{-\bk_j}\bs_j,\bar{\bx}_j\big)\bigg)\bigg), 
	\end{split}
	\end{equation*}
	where  $\bar{\bk}_j^*:= (n+1-|\bk_j|_1,0,\ldots,0) \in \NN_0^{d-j}$.
Notice that the functions
	\[
	f_{\bk_j,\bs_j}(\bar{\bx}_j):= (-1)^j2^{-j}2^{\alpha(j+   |\bk_j|_1)} \prod_{i=0}^j  \Delta_{2^{-k_i-1}}^2
	f\big(2^{-\bk_j}\bs_j,\bar{\bx}_j\big)
	\]
	are in variable $\bar{\bx}_j \in \II^{d-j}$ and  their   norm in $H^{\alpha}_{\infty}(\II^{d-j})$ (with respect to $\bar{\bx}_j$) satisfies the inequality
$
	\|f_{\bk_j,\bs_j}\|_{H^{\alpha}_{\infty}(\II^{d-j})} \leq  1.
$
Again, for ${T}_{\bar{\bk}_j^*} (f_{\bk_j,\bs_j})$  with $\bar{\bk}_j^* \in \NN_0^{d-j}$ we take ${S}_{\bar{\bk}_j^*,m}(f_{\bk_j,\bs_j})$ by formula \eqref{eq-S-Tkf} and apply \eqref{eq:Tkf} to have the estimate
	\begin{equation} \label{{T}_{bar{bk}_j^*}}
	\begin{aligned}
	\big\|{T}_{\bar{\bk}_j^*}  (f_{\bk_j,\bs_j}) 
	- {S}_{\bar{\bk}_j^*,m}(f_{\bk_j,\bs_j})\big\|_\infty 
	&
	\leq  (2B)^{d-j} 
	\big(2^{m} 2^{|\bar{\bk}_j^*|_1}\big)^{-\alpha}\binom{m+d-j}{d-j-1}.
	\end{aligned}
	\end{equation}
	For approximation of $F_{\bk_j}$, $j=0,\ldots,d-1,$ we take the functions $	S_{F_{\bk_j}}$ which are defined by the explicit formulas
	\begin{equation}\label{eq-SFkj}
	\begin{aligned}
	S_{F_{\bk_0}}(\bx)&:= {S}_{(n+1)\bee^1,m}(f)(\bx);
	\\
	S_{F_{\bk_j}}(\bx)&:=  \sum_{\bs_j \in Z(\bk_j)} 2^{-\alpha(j+   |\bk_j|_1)} \varphi_{\bk_j,\bs_j}(\bx_j) 
	{S}_{\bar{\bk}_j^*,m}(f_{\bk_j,\bs_j})(\bar{\bx}_j), \ j=1,\ldots,d-1.
	\end{aligned}
	\end{equation}
We have the estimates by \eqref{{T}_{bar{bk}_j^*}}
 for  every $ j=1,\ldots,d-1$ and $\bx \in \IId$,
	\begin{equation*}
	\begin{aligned}
	\big|F_{\bk_j}(\bx) - S_{F_{\bk_j}}(\bx)\big|
	& =
	\bigg|	\sum_{\bs_j \in Z(\bk_j)} 2^{-\alpha(j+   |\bk_j|_1)}\varphi_{\bk_j,\bs_j}(\bx_j)  
	\Big( {T}_{\bar{\bk}_j^*}  (f_{\bk_j,\bs_j}) 
	- {S}_{\bar{\bk}_j^*,m}(f_{\bk_j,\bs_j})\Big)(\bar{\bx}_j)\bigg|
	\\
	&
	\leq 
		 	\sum_{\bs_j \in Z(\bk_j)} 2^{-\alpha(j+   |\bk_j|_1)} \varphi_{\bk_j,\bs_j}(\bx_j)  
	\big\| {T}_{\bar{\bk}_j^*}  (f_{\bk_j,\bs_j}) 
	- {S}_{\bar{\bk}_j^*,m}(f_{\bk_j,\bs_j})\big\|_\infty
\\
& \leq  
(2B)^{d-j}
2^{-\alpha(j+   |\bk_j|_1)}\big(2^{m} 2^{|\bar{\bk}_j^*|_1}\big)^{-\alpha}\binom{m+d-j}{d-j-1}
	\\
&	= 
	2^{-\alpha}(2B)^{d-j}2^{-\alpha j}  2^{-(m+n)\alpha}\binom{m+d-j}{d-j-1}.
	\end{aligned}
	\end{equation*}
From the last estimate and \eqref{eq-F-k0} we deduce that
	\begin{equation*}
\begin{aligned}
\big\|F_{\bk_j}(\bx) - S_{F_{\bk_j}}(\bx)\big\|_\infty
& \leq 
2^{-\alpha}(2B)^{d-j}2^{-\alpha j}  2^{-(m+n)\alpha}\binom{m+d-j}{d-j-1}, \ \  j=0,1,\ldots,d-1.
\end{aligned}
\end{equation*}
Then we get
	\begin{equation*}
	\begin{aligned}
	\Bigg\| f-{R}_n(f) - \sum_{j=0}^{d-1} \sum_{|\bk_j| \le n} S_{F_{\bk_j}}\Bigg\|_\infty
  &\leq  
  \sum_{j=0}^{d-1}\  \sum_{|\bk_j| \le n} \big\|F_{\bk_j}(\bx) - S_{F_{\bk_j}}(\bx)\big\|_\infty
  \\
  &
  \leq  2^{-\alpha}     2^{-\alpha(m+n)}
	\sum_{j=0}^{d-1} (2B)^{d-j} 2^{-j\alpha} 
	\binom{m+d-j}{d-j-1}\binom{n+j}{j},
	\end{aligned}
	\end{equation*}
where  we have used 
$
\sum_{|\bk_j| \le n} 1=\binom{n+j}{j}
$. By the inequalities 
$
 	\binom{m+d-j}{d-j-1}\binom{n+j}{j} \leq \binom{m+n+d}{d-1}
$ for $j=0,\ldots,d-1$ and $B\geq 1$  the error of the approximation of $ f-{R}_n(f) $ by the function
\begin{equation*} \label{sum}
\sum_{j=0}^{d-1} \sum_{|\bk_j| \le n} S_{F_{\bk_j}}
\end{equation*}
can be estimated as 
\begin{equation} \label{error}
\begin{aligned}
\Bigg\| f-{R}_n(f) - \sum_{j=0}^{d-1} \sum_{|\bk_j| \le n} S_{F_{\bk_j}}\Bigg\|_\infty
& \leq\, 2^{-\alpha}    (2B)^{d}  2^{-\alpha(m+n)}\binom{m+n+d}{d-1}
\sum_{j=0}^{\infty} (2B)^{-j} 2^{-j\alpha} 
\\
&\leq 
\,  2^{-\alpha+1} (2B)^{d}\, 2^{-\alpha(m+n)}\binom{m+n+d}{d-1}.
\end{aligned}
\end{equation}

\noindent
	{\it Step 2.}  Due to \eqref{error}, to complete the proof  we explicitly construct a map of the form \eqref{lambda^j_{m,n}}  so that
	 \begin{equation} \label{= G_{m,n}^j(lambda^j_{m,n})}
	 \sum_{|\bk_j|_1 \le n} S_{F_{\bk_j}} = G_{m,n}^j(\blambda^j_{m,n}), \ j=0,\ldots,d-1.
	 \end{equation}
	  We deal with the cases $j\in \{1,\ldots,d-1\}$. The case $j=0$ is carried out similarly with slight modification. For $j\in \{1,\ldots,d-1\}$, with \eqref{eq-S-Tkf} and \eqref{S_f}, from \eqref{eq-SFkj}
	 we can write 
	\begin{equation*}
	\begin{aligned}
	& \sum_{|\bk_j|_1 \le n} S_{F_{\bk_j}}(\bx) 
	= \sum_{|\bk_j|_1 \le n} \sum_{\bs_j \in Z(\bk_j)}2^{-\alpha(j+   |\bk_j|_1)} \varphi_{\bk_j,\bs_j}(\bx_j) 
	{S}_{\bar{\bk}_j^*,m}(f_{\bk_j,\bs_j})(\bar{\bx}_j) 
	\\[1ex] 
	&= \sum_{|\bk_j|_1 \le n} \sum_{\bs_j \in Z(\bk_j)} \frac{2^{  d-j}}{2^{\alpha (j+|\bk_j|_1+|\bar{\bk}_j^*|_1)}}\varphi_{\bk_j,\bs_j}(\bx_j) \sum_{\bar{\bs}_j^*\in Z(\bar{\bk}_j^*)} {S}_m\big({{T}_{\bar{\bk}_j^*,\bar{\bs}_j^*}(f_{\bk_j,\bs_j})}\big)
	\big(2^{\bar{\bk}_j^*}\bar{\bx}_j-\bar{\bs}_j^*\big)
	\\[1ex] 
	&= \sum_{|\bk_j|_1 \le n} \sum_{\bs_j \in Z(\bk_j)} \frac{2^{  d-j}}{2^{\alpha(n+1+j)}}\varphi_{\bk_j,\bs_j}(\bx_j) 
	\\[1ex] 
	&\times 
	\sum_{s_{j+1}=0} ^{2^{n+1-|\bk_j|_1}}{S}_m\big({{T}_{(n+1-|\bk_j|_1,0,\ldots,0),(s_{j+1},0,\ldots,0)}(f_{\bk_j,\bs_j})}\big)
	\big(2^{n+1-|\bk_j|_1}x_{j+1} - s_{j+1},\bar{\bx}_{j+1}\big).
	\end{aligned}
	\end{equation*}
	Notice that by \eqref{T_k,s}  the functions ${T}_{(n+1-|\bk_j|_1,0,\ldots,0),(s_{j+1},0,\ldots,0)}(f_{\bk_j,\bs_j})$  belong to $\mathring{U}^{\alpha,d-j}_{\infty}$ and, therefore by Lemma \ref{lem:pattern},  the functions ${S}_m\big({{T}_{(n+1-|\bk_j|_1,0,\ldots,0),(s_{j+1},0,\ldots,0)}(f_{\bk_j,\bs_j})}\big)$ belong to $\Ss^{\alpha,d-j}(m)$. Numbering elements of $\Ss^{\alpha,d-j}(m)$ as
	$$
	\Ss^{\alpha,d-j}(m):= \big\{S_{\eta}^{d-j}\big\}_{\eta =1}^{N_{d-j}(m)},
	$$
	 we obtain
	\begin{equation*}
\begin{aligned}
\sum_{|\bk_j|_1 \le n} S_{F_{\bk_j}}(\bx) 
&= 
\sum_{\eta=1}^{N_{d-j}(m)}\sum_{(\bk_j,\bs_j,s_{j+1})\in \Gamma_\eta}
\frac{2^{  d-j}}{2^{\alpha(n+1+j)}}\varphi_{\bk_j,\bs_j}(\bx_j) 
S_{\eta}^{d-j}
\big(2^{n+1-|\bk_j|_1}x_{j+1} - s_{j+1},\bar{\bx}_{j+1}\big),
\end{aligned}
\end{equation*}
 where  
$$
\Gamma_\eta  =\Big\{ (\bk_j, \bs_j, s_{j+1}) \in  \Gamma_j(n): {S}_m\big({{T}_{(n+1-|\bk_j|_1,0,\ldots,0),(s_{j+1},0,\ldots,0)}(f_{\bk_j,\bs_j})}\big)=S_{\eta}^{d-j} 
\Big\}.
$$
We define the map
$$
\blambda^j_{m,n}:\ \Uas \to    \RR^{M_{d-j}(m)} \times [N_{d-j}(m)]^{|\Gamma_j(n)|}
$$ 
by
$$
f\ \mapsto  \blambda^j_{m,n}(f):= (\ba^\eta)_{\eta=1}^{N_{d-j}(m)} \times \big(\theta_{(\bk_j,\bs_j,s_{j+1})}(f)\big)_{(\bk_j,\bs_j,s_{j+1})\in \Gamma_j(n)},
$$
where 
$$
\ba^\eta
:= \ \big(a_{\bar{\bell}_j,\bar{\bt}_j}^{\eta}\big)_{ |\bar{\bell}_j|_1\leq m, \, \bar{\bt}_j \in Z(\bar{\bell}_j)}
$$
 are coefficients of $S_\eta^{d-j}$ in  Faber series representation  and $\theta_{(\bk_j,\bs_j,s_{j+1}) }(f) =\eta$ if  $(\bk_j,\bs_j,s_{j+1})\in \Gamma_\eta$. Hence we can write
\begin{equation*}
\begin{aligned}
\sum_{|\bk_j|_1 \le n} S_{F_{\bk_j}}(\bx) 
&= 
\sum_{\eta=1}^{N_{d-j}(m)}\sum_{(\bk_j,\bs_j,s_{j+1}): \atop
\theta_{(\bk_j,\bs_j,s_{j+1})}(f)=\eta}
\frac{2^{  d-j}}{2^{\alpha(n+1+j)}}\varphi_{\bk_j,\bs_j}(\bx_j) 
S_{\eta}^{d-j}
\big(2^{n+1-|\bk_j|_1}x_{j+1} - s_{j+1},\bar{\bx}_{j+1}\big)
\\[1.5ex]
& = G^j_{m,n}(\blambda^j_{m,n}(f)).
\end{aligned}
\end{equation*}
 This proves \eqref{= G_{m,n}^j(lambda^j_{m,n})}. 
\hfill
\end{proof}

Let the map
$$
G^R_n:\ \RR^{|\dim \Ff^d(n)|}\to \Ff^d(n)
$$ be defined  by $$
\blambda^R=(\lambda_{\bk,\bs})_{|\bk|_1\leq n,\bs\in Z(\bk)} \  \mapsto \ G^R_n(\blambda^R)=\sum_{|\bk|_1\leq n}\sum_{\bs\in Z(\bk)}\lambda_{\bk,\bs} \varphi_{\bk,\bs}.
$$
We extend the map $G^j_{m,n}$  defined in \eqref{G^j_{m,n}}  as a map  
	$$
	G^j_{m,n}: \, \RR^{M_{d-j}(m)} \times \RR^{|\Gamma_j(n)|}\to \Ff^d(m+n+1)
	$$ 
	(the extension denoted again by $G^j_{m,n}$)  by assigning 
	$G^j_{m,n}(\blambda^j)=0$ if 
	$$
	\blambda^j \not \in  \RR^{M_{d-j}(m)} \times [N_{d-j}(m)]^{|\Gamma_j(n)|}.
	$$
Denote 
\begin{equation}\label{eq-Nmn}
N_{m,n}:=\dim \Ff^d(n)+ \sum_{j=0}^{d-1}\big(M_{d-j}(m) +|\Gamma_j(n)|\big)
\end{equation}
and
$$
\blambda := \big(\blambda^R,\blambda^0,\ldots,\blambda^{d-1}\big) \in \RR^{N_{m,n}},
$$
where $\blambda^R \in \RR^{|\dim \Ff^d(n)|}$ and  
$\blambda^j \in \RR^{M_{d-j}(m)} \times \RR^{|\Gamma_j(n)|}$. 
We define the map 
	\begin{equation} \label{G_{m,n}1}	
G_{m,n}: \RR^{N_{m,n}} \to \Ff^d(m+n+1)
	\end{equation}
 by
	\begin{equation} \label{G_{m,n}2}	
G_{m,n}(\blambda):=G^R_n(\blambda^R)+\sum_{j=0}^{d-1}G_{m,n}^j(\blambda^j),
	\end{equation}
and put $K_{m,n}:=\dim (\Ff^d(m+n+1))$.

\begin{corollary} \label{corollary:Nwidth}
	Let $\alpha\in (0,1]$, $d , m,n\in \NN$.    Then we can explicitly construct a map
		$$
		\blambda_{m,n}:\ \Uas\ \to  \  \RR^{N_{m,n}}
		$$
		 so that  
		\begin{equation} \label{(i)}	
		\begin{aligned}
		 \sup_{f\in \Uas}\|f-G_{m,n}( \blambda_{m,n}(f))\|_\infty
		\leq  2^{-\alpha+1}(2B)^{d}  2^{-\alpha(m+n)}\binom{m+n+d}{d-1},
		\end{aligned}
		\end{equation}
		and hence,
		\begin{equation*} 	
		\begin{aligned}
		d_{K_{m,n}}\big(\Uas,L_\infty\big)
		\leq  2^{-\alpha+1}(2B)^{d}  2^{-\alpha(m+n)}\binom{m+n+d}{d-1}.
		\end{aligned}
		\end{equation*}		
\end{corollary}

\begin{proof} We define the operator $\blambda_{m,n} $ by
	$$
	\blambda_{m,n} := \big(\blambda^R_n,\,\blambda^0_{m,n},\ldots,\,\blambda^{d-1}_{m,n}\big),
	$$
	where the operators $\blambda^j_{m,n}$ are as in Lemma  \ref{prop-f-Rn-Gnm} and the operator $\blambda^R_n$ is defined by 
	$$
	\blambda^R_n(f):= (\lambda_{\bk,\bs}(f))_{|\bk|_1\leq n,\bs\in Z(\bk)}.
	$$	
	Then the  upper bound is already proved in Lemma \ref{prop-f-Rn-Gnm}. For the lower bound, by definition of Kolmogorov width we derive that
	\begin{align} \label{lowerbound}
	d_{K_{m,n}}(\Uas,L_\infty)
	\leq \sup_{f\in \Uas}\|f-G_{m,n}(\blambda_{m,n}(f))\|_\infty.
	\end{align}
\hfill
\end{proof}

\begin{lemma} \label{N_{m,n}}
For $n,m,d \in \NN$ its holds the inequality	
\begin{align*}
N_{m,n}
\ \leq \ 3^{2^{m+1}\binom{m+d-1}{d-1}} 2^{m+1} \binom{m+d}{d-1} + 2^{n+2}\binom{n+d}{d-1}.
\end{align*}	
\end{lemma}
\begin{proof}
 From \eqref{eq-Nmn}, \eqref{eq-Mdj}, and \eqref{eq-Gammajn} we have that
\begin{align*}
N_{m,n}
&\leq  \dim \Ff^d(n)+ \sum_{j=0}^{d-1} \bigg( 3^{2^{m+1}\binom{m+d-j-1}{d-j-1}} 2^{m+1}\binom{m+d-j-1}{d-j-1} + 2^{n+1}\binom{n+j}{j} \bigg)
\\
&\leq  \sum_{\ell=0}^n 2^\ell\binom{\ell+d-1}{d-1}+  3^{2^{m+1}\binom{m+d-1}{d-1}}2^{m+1}\sum_{j=0}^{d-1}\binom{m+d-j-1}{d-j-1} + 2^{n+1}\sum_{j=0}^{d-1} \binom{n+j}{j} 
\\
& \leq  2^{n+1}\binom{n+d-1}{d-1}+3^{2^{m+1}\binom{m+d-1}{d-1}} 2^{m+1} \binom{m+d}{d-1} + 2^{n+1}\binom{n+d}{d-1}
\\
& 
\leq 3^{2^{m+1}\binom{m+d-1}{d-1}} 2^{m+1} \binom{m+d}{d-1} + 2^{n+2}\binom{n+d}{d-1},
\end{align*}
where in the third estimate we have used  $\sum_{j=k}^{\ell}\binom{j}{k}=\binom{\ell+1}{k+1}$. 	
\hfill	
\end{proof}

We now are able to explicitly construct such a method $Q_N(f) = G^*_N((\blambda^*_N(f)))$ with mappings $\blambda^*_N$ and $G^*_N$ of the form \eqref{lambda_N,G_N}, satisfying the upper and lower estimates of approximation error \eqref{upperbnd}--\eqref{lowerbnd}.  Recall that $\Ff^d(m)$ is  the finite-dimensional subspace in  $\mathring{C}(\IId)$ of the form \eqref{Ff^d(m)} and that $R_m(f) \in \Ff^d(m) $ for $f \in \mathring{C}(\IId)$, and
$\dim \Ff^d(m) = \sum_{\ell=0}^m 2^\ell\binom{\ell+d-1}{d-1}$.
\begin{theorem} \label{thm:Nwidth}
Let $\alpha\in (0,1]$, $d \in \NN$. Then  for every
 	\begin{equation} \label{N>}
 N \, \geq \, N(d):= \ 3^{2^{d+2}\binom{2d}{d-1}} 2^{d+3}\binom{2d+1}{d-1}, \ N\in \NN,
 \end{equation}
  we can explicitly determine a number $m^*(N) \le \log N + \log\log N +1$, $m^*(N) \in\NN$,  and explicitly construct maps
	$$
	\blambda^*_N:\ \Uas \to    \RR^N\qquad {\text and}\qquad G^*_N:\ \RR^N \to \Ff^d\big( m^*(N)\big)
	$$ 
	so that   $N \, \le \,  M(N):= \dim \Ff^d\big(m^*(N)) $,
	\begin{equation} \label{dim}
	\frac{((d-1)!)^2}{2^{7}(4d-4)^{d-1}}  N  \log N (\log\log N)^{1-d}	 \leq   M(N)
	\le  \frac{(12d^3)^{d-1}}{(d-1)!} N(\log N)(\log\log N)^{1-d}
	\end{equation}
	and
\begin{equation} \label{(ii)}	
\begin{aligned}
	\sup_{f\in \Uas}\|f-G_N^*(\blambda^*_N(f))\|_\infty
\leq 
C_\alpha \left(\frac{K^{d-1}}{(d-1)!}\right)^{2\alpha +1} \frac{(\log N)^{(d-1)(\alpha +1)}}{(N\log N)^{\alpha} }  (\log\log N)^{(d-1)\alpha},
\end{aligned}
\end{equation}	
where $K:= \left(4^\alpha 6 /(2^{\alpha}-1)\right)^{1/(2\alpha +1)}$, $C_\alpha:=   2^{7\alpha+2}/(2^{\alpha}-1)$.  
Moreover, if $\alpha < 1$,
\begin{equation} \label{thm:lowerbound}
\sup_{f\in \Uas}\|f-G_N^*(\blambda^*_N(f))\|_\infty
\, \ge \, 	d_{ M(N)}(\Uas,L_\infty) \, \ge \,
C_{d,\alpha}\,	\frac{(\log N)^{(d-1)(\alpha +\frac{1}{2})}}{(N\log N)^{\alpha} }  (\log\log N)^{(d-1)\alpha}.
\end{equation}	
\end{theorem}

\begin{proof} We prove the case $d\geq 2$. The case $d=1$ is carried out similarly. Fix a number $N \in \NN$ satisfying the condition \eqref{N>}. We define 
\begin{equation*} 
m^*(N):= n(N) + m(N),
\end{equation*}
where $n = n(N)$ and $m= m(N)$ are chosen so that
\begin{align} \label{n:N/2}
2^{n+2}\binom{n+d}{d-1}\leq \frac{N}{2} < 2^{n+3}\binom{n+d+1}{d-1}
\end{align}
and
\begin{align*} \label{m:N/2}
3^{2^{m+1}\binom{m+d-1}{d-1}} 2^{m+1}\binom{m+d}{d-1} \leq \frac{N}{2}< 3^{2^{m+2}\binom{m+d}{d-1}} 2^{m+2}\binom{m+d+1}{d-1}.
\end{align*}
From this choice of $m,$ $n$ and Lemma \ref{N_{m,n}} we can see that 
$
N_{m,n} \, \le \, N.
$
This allows us to define 
$
\blambda^*_N:= \blambda_{m,n}$ and  $G^*_N$ as an extension of $G_{m,n}$ from $\RR^{N_{m,n}}$ to $\RR^N$ with the chosen $m,n$, where  
$$
\blambda_{m,n}:\ \Uas \to    \RR^{N_{m,n}}
$$ 
 is as in Corollary \ref{corollary:Nwidth} and  
 $$
 G_{m,n}: \ \RR^{N_{m,n}} \to \Ff^d(m+n+1)
 $$ 
 as in \eqref{G_{m,n}1}--\eqref{G_{m,n}2}.

Let us first prove  the dimension-dependent upper estimate \eqref{(ii)}.
The choice of $m$, $n$ and the assumption \eqref{N>}   
 implies $n \geq m \geq d+1$. With $n\geq d+1$ we have the estimate
\begin{equation}\label{eq-estimate-n}
2^3 N^{-1}\frac{n^{d-1}}{(d-1)^{d-1}}\leq 2^{-n}< 2^4 N^{-1}\binom{n+d+1}{d-1}< 2^4 N^{-1} \frac{(2n)^{d-1}}{(d-1)!}\quad\text{and}\quad  n\leq \log N \leq 4d n.
\end{equation}
With $m\geq d+1$ we deduce
$$
3^{2^{m+1}\binom{m+d-1}{d-1}} 2^{m+1}\binom{m+d}{d-1} \leq \frac{N}{2}<
3^{2^{m+2}\binom{m+d+1}{d-1}} 2^{m+2}\binom{m+d+1}{d-1} \leq \frac{1}{2}4^{2^{m+2}\binom{m+d+1}{d-1}},
$$
where we have used $3^tt\leq \frac{1}{2}4^t$ for $t\geq 16$. Hence,
\begin{align*}
2^{m+1}\binom{m+d-1}{d-1}\log 3\leq \log N< 2^{m+4}\binom{m+d+1}{d-1}
\end{align*}
which implies  
\begin{equation}\label{eq-estimate-m}
2\log 3(\log N)^{-1} \frac{m^{d-1}}{(d-1)^{d-1}} \leq  2^{-m}\leq 2^4 (\log N)^{-1} \frac{(2m)^{d-1}}{(d-1)!}\quad \text{and}\quad  m\leq \log\log N\leq 4dm.
\end{equation}
Consequently, we obtain
\begin{align*}
2^{-\alpha(m+n)}\binom{m+n+d}{d-1} & \leq 2^{-\alpha(m+n)}\frac{(3n)^{d-1}}{(d-1)!}
\\[1.5ex]
& \leq 
\bigg(  2^4 N^{-1} \frac{(2\log N)^{d-1}}{(d-1)!}\bigg)^{\alpha} \bigg(2^4 (\log N)^{-1} \frac{(2\log\log N)^{d-1}}{(d-1)!} \bigg)^{\alpha}\frac{(3\log N)^{d-1}}{(d-1)!}
\\[1.5ex]
&\leq \frac{2^{8\alpha}(4^\alpha 3)^{d-1}}{((d-1)!)^{2\alpha +1}}\frac{(\log N)^{(d-1)(\alpha +1)}}{(N\log N)^{\alpha} }  (\log\log N)^{(d-1)\alpha}.
\end{align*}
 This together with the upper bound \eqref{(i)} proves the upper bound \eqref{(ii)}. 
 
Next, we prove \eqref{dim}.  With $ M(N)=\dim(\Ff^d(m^*(N)))$, from the choice of $n$ as in \eqref{n:N/2}  we obtain
\begin{equation*}
\begin{aligned}
N \,& \leq  \, 2^{n+4}\binom{n+d+1}{d-1}-1 \, \le \, \sum_{\ell=0}^{m+n+1}2^\ell\binom{\ell+d-1}{d-1}
\  = \  M(N).
\end{aligned}
\end{equation*}
Moreover, taking account $n \geq m \geq d+1$ we derive that
\begin{equation*}
\begin{aligned}
2^{m+n+1} \Big(\frac{n}{d-1}\Big)^{d-1}	\leq 2^{m+n+1}\binom{m+n+d}{d-1}
\leq   M(N)
\leq \, 2^{m+n+2}\binom{m+n+d}{d-1}
\leq \frac{3^{d-1}}{(d-1)!}  2^{n+m+2}n^{d-1}  
\end{aligned}
\end{equation*}
which from  \eqref{eq-estimate-n} and \eqref{eq-estimate-m}   implies
\begin{equation*}
\begin{aligned}
\frac{((d-1)!)^2}{2^{7}(4d-4)^{d-1}}    \frac{N\log N}{m^{d-1}}	\leq  	 M(N)
\leq \frac{3^{d-1}}{(d-1)!}  (d-1)^{2d-2}   \frac{N\log N}{m^{d-1}}
\end{aligned}
\end{equation*}
and therefore, \eqref{dim}.

We finally verify the lower bound \eqref{thm:lowerbound}.    
From the known inequality 
$$d_M(\Uas,L_\infty) \gtrsim  M^{-\alpha} (\log M)^{(d-1)(\alpha+\frac{1}{2})},$$ (see, e.g., \cite[Theorem 4.3.11]{DTU18B}) we obtain that
	\begin{align*}
	d_{ M(N)}(\Uas,L_\infty)
	& \geq C_{\alpha,d}  ( M(N))^{-\alpha} (\log  M(N))^{(d-1)(\alpha+\frac{1}{2})}
	\\
	&\geq C_{\alpha,d}  \bigg(\frac{N\log N}{(\log\log N)^{d-1}} \bigg)^{-\alpha} (\log N)^{(d-1)(\alpha+\frac{1}{2})}
	\\
	& \geq C_{\alpha,d} \frac{(\log N)^{(d-1)(\alpha +\frac{1}{2})}}{(N\log N)^{\alpha} }  (\log\log N)^{(d-1)\alpha}.
	\end{align*}
  Now provided with $M(N) = K_{m(N),n(N)}$ the lower bound  \eqref{thm:lowerbound} follows from \eqref{lowerbound}.
\hfill
\end{proof}


 From the left inequality in \eqref{d_M<d_N,M}
 and Theorem \ref{thm:Nwidth} we deduce the following upper and lower bounds for $d_ {N,M(N)} (\Uas,L_\infty)$.
\begin{corollary} \label{cor:d_NM}
	 Let $\alpha\in (0,1)$, $d \in \NN$ and  $N \ge 4$. With $M(N)=\lfloor N(\log N)(\log\log N)^{-(d-1)}\rfloor$  we have
		\begin{equation*}  
		\begin{aligned}	\frac{(\log N)^{(d-1)(\alpha +\frac{1}{2})}}{(N\log N)^{\alpha} }  (\log\log N)^{(d-1)\alpha}
 \lesssim 
	 d_ {N,M(N)} (\Uas,L_\infty)
 \lesssim
	 \frac{(\log N)^{(d-1)(\alpha +1)}}{(N\log N)^{\alpha} }  (\log\log N)^{(d-1)\alpha}.
		\end{aligned}
		\end{equation*}
\end{corollary}

In the case when $d=2$ we get the right asymptotic order of  $d_ {N,M(N)} (\Uastwo,L_\infty(\II^2))$  as in the following theorem.

\begin{theorem} 
Let $\alpha\in (0,1)$.  With $M(N):=  \lfloor N(\log N)(\log\log N)^{-1}\rfloor$   we have
		\begin{equation} \label{(iii)}	
		d_ {N,M(N)}  (\Uastwo,L_\infty(\II^2))
	\asymp 
		\sup_{f\in \Uastwo}\|f-G_N^*(\blambda^*_N(f))\|_\infty
		 \asymp  
	N^{-\alpha}	\log N  (\log\log N)^{\alpha}.
		\end{equation}		 
\end{theorem}
\begin{proof}
From \eqref{(ii)} we immediately get the upper bound in \eqref{(iii)}:	
	\begin{equation} \nonumber	
d_ {N,M(N)} (\Uastwo,L_\infty(\II^2))
\lesssim
\sup_{f\in \Uastwo}\|f-G_N^*(\blambda^*_N(f))\|_\infty
\lesssim 
N^{-\alpha}	\log N  (\log\log N)^{\alpha}.
\end{equation}		 		
To prove the lower bound we use the known asymptotic order of the Kolmogorov $M$-widths for $\alpha \in (0,1)$
$$
d_M(\Uastwo,L_\infty(\II^2))\asymp M^{-\alpha}(\log M)^{\alpha+1},
$$
see, e.g., \cite[Theorem 4.3.14]{DTU18B} and references there.  Hence, with
 $n = n(N)$ and $m= m(N)$ defined as in the proof of Theorem \ref{thm:Nwidth} for $d=2$ we derive the lower bound:
 \begin{align}\nonumber
 \sup_{f\in \Uastwo}\|f-G_N^*(\blambda^*_N(f))\|_\infty
 \ &\ge \ 
d_ {N,M(N)} (\Uastwo,L_\infty(\II^2))
\ \ge \ d_ {M(N)} (\Uastwo,L_\infty(\II^2))
\\
 & \geq C_\alpha \, M(N)^{-\alpha}(\log M(N))^{\alpha+1} \nonumber
\geq C_\alpha  \, 	N^{-\alpha}	\log N  (\log\log N)^{\alpha}.
 \end{align}
\hfill	
\end{proof}

\noindent
{\bf Acknowledgments.}  This work is funded by Vietnam National Foundation for Science and Technology Development (NAFOSTED) under  Grant No. 102.01-2020.03. A part of this work was done when  the authors were working at the Vietnam Institute for Advanced Study in Mathematics (VIASM). They would like to thank  the VIASM  for providing a fruitful research environment and working condition.

\bibliographystyle{abbrv}

\bibliography{AllBib}

\end{document}